\documentclass[12pt]{amsart}

\usepackage[utf8]{inputenc}

\usepackage{geometry}
\usepackage{amsmath,amssymb,amsfonts,amsthm,amsopn,bbm}
\usepackage{graphicx,tikz,forest}
\usepackage{pgfplots}
\usepackage[all]{xy}
\usepackage{amsmath}
\usepackage{multirow, longtable, makecell, array}
\usepackage[shortlabels]{enumitem}
\usepackage{amssymb}
\usepackage{mathtools}
\usepackage{bookmark}
\usepackage{hyperref}

\newtheorem{theorem}{Theorem}[section]
\newtheorem{lemma}[theorem]{Lemma}

\newtheorem{proposition}[theorem]{Proposition}

\allowdisplaybreaks

\newtheorem{remark}[theorem]{Remark}

\theoremstyle{definition}
\usepackage{microtype}

\usepackage{cellspace} %
\usepackage{caption}
\usepackage{subcaption}

\newcommand{\fm}{f_{\leq 2^{j-k}}}
\newcommand{\fp}{f_{\geq 2^{j+k}}}
\newcommand{\fii}{f_{2^{j-k}\leq \cdot \leq 2^{j+k}}}

\newcommand{\gm}{g_{\leq 2^{j-k}}}
\newcommand{\gp}{g_{\geq 2^{j+k}}}
\newcommand{\gii}{g_{2^{j-k}\leq \cdot \leq 2^{j+k}}}

\def\Ca{C^{\alpha}}
\def\bR{\mathbb{R}}
\def\bC{\mathbb{C}}
\def\bN{\mathbb{N}}
\def\bZ{\mathbb{Z}}

\def\cF{\mathcal{F}}
\def\cP{\mathcal{P}}

\def\cA{\mathcal{A}}
\def\cB{\mathcal{B}}

\def\cH{\mathcal{H}}

\def\rd{\bR^d}

\def\rdd{\bR^{2d}}

\def\lc{\left(}
\def\rc{\right)}
\def\lV{\Big\Vert}
\def\rV{\Big\Vert}

\def\wt{\widetilde}
\def\*b{*_{\bullet}}

\def\Bd'{B_{\delta'}}

\def\cBd'{\bar{B}_{\delta'}}

\def\ird{\int_{\rd}}

\def\wh{\widehat}

\def\Lip{{\rm Lip}}
\def\badmax{\max\Big\{ \log \frac{\|\Delta \tau\|_{L^\infty}}{\|D\tau \|_{L^\infty}}, 1 \Big\} }

\newcommand{\adm}[1]{{\left\Vert #1 
		\right\Vert}}

\newcommand\scalemath[2]{\scalebox{#1}{\mbox{\ensuremath{\displaystyle #2}}}}

\numberwithin{equation}{section}
\makeatletter
\newcommand{\pushright}[1]{\ifmeasuring@#1\else\omit\hfill$\displaystyle#1$\fi\ignorespaces}
\newcommand{\pushleft}[1]{\ifmeasuring@#1\else\omit$\displaystyle#1$\hfill\fi\ignorespaces}
\makeatother
\begin{document}
	
\title[On the stability of the scattering transform]{Stability of the scattering transform for deformations with minimal regularity}
\author{Fabio Nicola}
\address{Dipartimento di Scienze Matematiche, Politecnico di Torino, Corso Duca degli Abruzzi 24, 10129 Torino, Italy.}
\email{fabio.nicola@polito.it}
\author{S. Ivan Trapasso}
\address{MaLGa Center - Department of Mathematics (DIMA), Università di Genova. Via Dodecaneso 35, 16146 Genova, Italy.}
\email{salvatoreivan.trapasso@unige.it}

\begin{abstract} 
Within the mathematical analysis of deep convolutional neural networks, the wavelet scattering transform introduced by St\'ephane Mallat is a unique example of how the ideas of multiscale analysis can be combined with a cascade of modulus nonlinearities to build a nonexpansive, translation invariant signal representation with provable geometric stability properties, namely Lipschitz continuity to the action of small $C^2$ diffeomorphisms -- a remarkable result for both theoretical and practical purposes, inherently depending on the choice of the filters and their arrangement into a hierarchical architecture. In this note, we further investigate the intimate relationship between the scattering structure and the regularity of the
deformation in the H\"older regularity scale $C^\alpha$, $\alpha >0$. We are able to precisely identify the stability threshold, proving that stability is still achievable for deformations of class $C^{\alpha}$, $\alpha>1$, whereas instability phenomena can occur at lower regularity levels modelled by $C^\alpha$, $0\le \alpha <1$. While the behaviour at the threshold given by Lipschitz (or even $C^1$) regularity remains beyond reach, we are able to prove a stability bound in that case, up to $\varepsilon$ losses. 
\end{abstract}
\maketitle

\section{Introduction}

Broadly speaking, the last decade was certainly marked by a striking series of successes in several machine learning tasks relying on neural networks \cite{lecun}. In particular, impressive results in image classification, pattern recognition and feature extraction were achieved by means of deep convolutional neural networks. Borrowing from Wigner, the efforts of many researchers are currently directed to provide explanations for the ``unreasonable effectiveness'' of these models and related intriguing phenomena, such as the double descent error curve \cite{belkin,hastie,mei,nakkiran} or the instability to adversarial attacks \cite{alaifari,carliniwag,goodfe,kurakin,sze}. 

The mathematical analysis of convolutional neural networks is a wide area of current interest in the literature. The present note fits into a line of research pioneered by St\'ephane Mallat, ultimately aimed at showing how some fundamental principles of harmonic analysis can be used to obtain theoretical models and guarantees in connection with problems of deep learning. Motivated by some properties naturally expected to be satisfied by a proper feature extractor, in the fundamental contribution \cite{mall cpam} it is shown how such conditions essentially force the design of a multiscale signal representation to have a hierarchical architecture that shares many similarities with that of a convolutional neural network. 

Let us briefly retrace here the basic ideas behind the construction for the sake of clarity. Motivated by image analysis, the goal is to build up a \textit{feature map} $\Phi \colon L^2(\rd) \to \cH$, with values in a suitable Hilbert space $\cH$, such that: 
\begin{enumerate}
    \item \textit{$\Phi$ is a nonexpansive transform}. \\ This condition ensures stability to additive perturbations, that is 
    \[ \adm{\Phi(f) - \Phi(h)} \le \| f-h \|_{L^2}, \quad f,h \in L^2(\rd). \]
    
    \item \textit{$\Phi$ is a translation-invariant transform}. \\ Let $L_x$ be the translation operator by $x \in \rd$, acting on $f \in L^2(\rd)$ as $L_x f(y) = f(y-x)$. Then 
    \[ \Phi(L_x f) = \Phi(f), \quad  f\in L^2(\rd),\quad x \in \rd. \]
    
    \item \textit{$\Phi$ is stable to the action of small diffeomorphisms.} \\ A convenient linearization of the action of a diffeomorphism along the orbits of the translation group leads one to consider deformation operators of the form $L_\tau f (y) \coloneqq f(y-\tau(y))$ with distortion field $\tau \colon \rd \to \rd$. Stability is achieved if the feature vectors of $L_\tau f$ and $f$ are close when the underlying diffeomorphism $1-\tau$ is close to identity, namely if there exists $C>0$ such that
    \[ \adm{\Phi(L_\tau f) - \Phi(f)} \le C K(\tau) \|f \|_{L^2}, \quad f \in L^2(\rd)\] where $K(\tau)$ is some complexity measure/cost associated with the deformation $\tau$.
    
    
\end{enumerate}

\subsection{The wavelet scattering transform}
The approach in \cite{mall cpam} relies on the a priori exploitation of the principles of multiscale analysis in order to satisfy the requirements detailed above. It is indeed well understood that instability to deformations is mostly attributable to the vulnerability of the high-frequency components of a signal, which however carry fine-structure details and cannot thus be discarded without deteriorating the information captured by the representation $\Phi$. A Littlewood-Paley wavelet transform \cite{mallbook, meyer} can be used to perform scale separation and rearrange the frequency content of a signal into dyadic packets. Thanks to inherent redundancy and additional nonlinear operations, this procedure allows one to stabilize the high-frequency content up to a certain scale, as well as to obtain stability guarantees to relatively small translations. Recovery of the information content discarded by a fixed scale wavelet transform is achieved by iteration of the same procedure on the outputs of the latter, ultimately leading to a cascade of convolutions with fixed wavelet filters and modulus nonlinearities that eventually has the multilayer architecture of a convolutional neural network. The pooling stage is performed by extracting low-frequency averages of each scattered wavelet coefficient, and actually coincides with output feature generation. 

An essential yet more detailed discussion of this construction is provided in Section \ref{sec notation}, where we also fix the notation used below. Here we just recall that a low-pass filter $\phi$ and a mother wavelet $\psi$ on $\rd$ are primarily chosen in such a way that the collection $\{\phi_{2^J}\} \cup \{\psi_\lambda\}_{\lambda \in \Lambda_J}$ obtained by suitable rotations and dilations up to the scale $2^J$, $J \in \bZ$ (see \eqref{eq def lambdaj} for the precise definition of the index set $\Lambda_J$), allow one to essentially cover the frequency space without holes -- as entailed by the Littlewood-Paley condition \eqref{eq LP} below. The wavelet modulus coefficient corresponding to $\lambda \in \Lambda_J$ is given by $U[\lambda]f = |f*\psi_\lambda|$. The cascading sequence that we mentioned before is obtained by iteration along all the possible paths with finite length, namely $\cP_J = \bigcup_{m\ge 0} \Lambda_J^m$, so that given $p=(\lambda_1, \ldots, \lambda_m) \in \Lambda_J^m$ we set
\[ U[p]f \coloneqq U[\lambda_m] \cdots U[\lambda_1]f. \] The windowed wavelet scattering transform at scale $2^J$ is thus the collection (indexed by $\cP_J$) of features obtained by averaging with the low-pass filter $\phi_{2^J}$ at the scale $2^J$: \[ S_J[\cP_J] f \coloneqq \{S_J[p] f\}_{ p \in \cP_J}, \quad S_J[p]f \coloneqq U[p] f * \phi_{2^J}. \]
The feature space corresponds to $\cH = \ell^2(\cP_J;L^2(\rd))$, hence
\[ \| S_J[\cP_J] f\|^2  = \sum_{p \in \cP_J} \| S_J[p]f \|_{L^2}^2. \]


Concerning the stability to small deformations, it was proved in \cite[Theorem 2.12]{mall cpam} that, under suitable assumptions on the frequency filters (see Section \ref{sec a brief review} below for details), for  every input signal $f$ with finite mixed $\ell^1 L^2$ scattering norm, that is
\[ \|U[\mathcal{P}_J] f\|_1 \coloneqq \sum_{m \ge 0} \Big( \sum_{p \in \Lambda_J^m} \| U[p] f\|^2_{L^2} \Big)^{1/2} < \infty, \] and for every deformation $\tau\in C^2(\rd;\rd)$ with $\|D\tau\|_{L^\infty}\leq 1/2$, the following stability estimate holds:
\begin{equation}\label{eq mallat0}
        \| S_J[\cP_J](L_\tau f) - S_J[\cP_J](f) \|  \le  C K_2(\tau) \|U[\cP_J]f\|_1, 
\end{equation} with
\[ K_2(\tau) =  2^{-J}\|\tau\|_{L^\infty} + \badmax \|D\tau\|_{L^\infty} + \|D^2\tau\|_{L^\infty}, \] where $\| \Delta \tau \|_{L^\infty} \coloneqq \sup_{x,y \in \rd} |\tau(x) - \tau(y)|$ and $D^2\tau$ stands for the Hessian of $\tau$. 

Some remarks are in order here. First, this estimate implies stability under small $C^2$ deformations, as well as approximate invariance to global translations up to the scale $2^J$ (with global invariance recaptured in the asymptotic regime $J\to +\infty$). 

Concerning the occurrence of the scattering norm, 
it is proved in \cite[Lemma 2.8]{mall cpam} that a similar $\ell^2 L^2$ norm is finite for functions with a certain average modulus of continuity in $L^2$, in particular for functions with logarithmic-Sobolev regularity. It is also worthwhile to point out that numerical evidences of exponential decay of the scattering energy coefficients were rigorously confirmed (at least in dimension $d=1$) in \cite{Waldspurger}. The latter results also imply that $1$-dimensional signals with a (generalized) logarithmic-Sobolev regularity have indeed finite $\ell^1 L^2$ scattering norm (see Proposition \ref{pro embed} below). 

It should be highlighted that one can also restrict to more regular signal classes, such as Sobolev spaces or band-limited and cartoon functions. The underlying gain in signal regularity usually comes along with some degree of stability to small deformations -- namely, $L^2$ sensitivity bounds of the form $\|L_\tau f - f\|_{L^2} = O (K(\tau))$ are satisfied for suitably small and regular deformations, see e.g.\ \cite{wiat paper,wiat old}. In view of the  Lipschitz continuity of the feature extractor, the latter bounds reflect into stability results for the signal representation, in a sense ``inherited'' from the sensitivity to deformations of the underlying signal class  \cite{balan, bietti, czaja,zou}. 
On the other hand, the estimate \eqref{eq mallat0} entails the more difficult problem of deriving ``structural'' stability guarantees from the very design of the feature extractor, which are thus informative on the invariance of the signal representation rather than the regularity of the signal itself. 

Putting aside these complementary views on the issue, let us observe that
the condition $\|D \tau\|_{L^\infty}\leq 1/2$ suffices to ensure that $I-\tau$ is a bi-Lipschitz map and $L_\tau \colon L^2(\rd)\to L^2(\rd)$ is well-defined and uniformly bounded\footnote{Indeed, for every $y\in\rd$, the map $\rd\to\rd$ given by $x\mapsto y+\tau(x)$ is a contraction, with Lipschitz constant $L\leq 1/2$. The map that associates $y$ with the corresponding unique fixed point $x$ has Lipschitz constant $\leq 1/(1-L)\le 2$.}. More precisely, if $f$ is concentrated in a certain frequency dyadic band, $L_\tau f$ is essentially concentrated in the same band as well and this suggests that such deformations should interact well with the transform architecture, which is adjusted to such bands by design. Nevertheless, we will see that some instability phenomena may occur if $\|D\tau\|_{L^\infty}\not\to0$. 

\subsection{A regularity scale for deformations}


The purpose of this note is to elucidate the intimate relationship between the scattering architecture and the regularity of the deformation, lying at the very core of the ``structural'' stability for the wavelet scattering transform. To this aim, we consider distortion fields in the $C^\alpha$  regularity scale, $\alpha>0$ (H\"older classes, recalled in Section \ref{sec notation} below), hence encompassing the case $\alpha=2$ already studied in \cite{mall cpam}. The quest for the minimal deformation regularity needed to achieve stability guarantees is an intriguing and natural challenge from a mathematical point of view, further motivated by the current practice in several problems in PDEs and image analysis \cite{scherzer,trouve,younes} where diffeomorphisms with lower regularity are taken into account -- for instance, Sobolev deformations $\tau \in H^s(\rd;\rd)$ with $s>d/2+1$, hence in $C^{s-d/2}(\rd;\rd)$. Stability results for the scattering transform under such weaker regularity assumptions for the deformation would then broaden the theoretical and practical scope of this mathematical theory, hence promoting cross-fertilisation with classical and recent problems arising in signal analysis and deep learning. 

As a first result we highlight the following instability phenomenon, when $\tau\to0$ in $C^\alpha$, for $0\leq \alpha<1$, but not in the $C^1$ norm. 
We assume here $d=1$. 

\begin{theorem}\label{thm mainthm0}  
 Suppose that the filters $\phi,\psi\in L^1(\bR)\cap L^2(\bR)$ in the definition of the scattering transform satisfy the Littlewood-Paley condition \eqref{eq LP} below. Assume, in addition, that
  $\psi$ has Fourier transform $\widehat{\psi}$ compactly supported in $(0,+\infty)$.

There exist $\tau,f\in C^\infty(\bR;\bR)\setminus\{0\}$ with compact support and satisfying $\|\tau'\|_{L^\infty}\le 1/2$ such that the following holds true. 

There exists $C>0$ such that, for every $J\in\bZ$, $n\in\bN$, setting $f_n(x)=2^{n/2}f(2^n x)$ and $\tau_n(x)=2^{-n}\tau(2^n x)$,  
\begin{equation}\label{eq cont 0}
\|S_J[\cP_J](L_{\tau_n} f_n)-S_J[\cP_J](f_n)\|\geq C.
\end{equation}
As a consequence, for $0\leq\alpha<1$, there exists $C>0$ such that, for every $J\in\bZ$, $n\in\bN$,
\begin{equation}\label{eq cons0}
\|S_J[\cP_J](L_{\tau_n} f_n)-S_J[\cP_J](f_n)\|\geq  C 2^{n(1-\alpha)}\|\tau_n\|_{C^\alpha}\|f_n\|_{L^2}.
\end{equation}
\end{theorem}
Results in the same spirit hold as well if the $L^2$-norm is replaced by the scattering norm $\|U[\cP_J]f\|_1$, see Proposition \ref{pro pro0} below.

Notice that the functions $\tau_n\in C^\infty(\bR;\bR)$ are all  supported in a fixed compact interval, say $I\subset\bR$. Moreover, $\|\tau'_n\|_{L^\infty}\le 1/2$ for every $n$, and $\|\tau_n\|_{C^\alpha}\to 0$ as $n\to\infty$ for every $0\leq\alpha<1$ (by \eqref{eq interp} below). 

To better frame the previous result, consider the set
\begin{equation} \label{eq def cB}
\mathcal{B}_{1/2}=\{\tau\in C^\infty(\mathbb{R};\mathbb{R}): {\rm supp}\,(\tau)\subset I,\ \|\tau'\|_{L^\infty}\le 1/2\},
\end{equation}
equipped with the $C^\alpha$ metric\footnote{On $\cB_{1/2}$ the $C^\alpha$ \textit{topology}, $0\leq\alpha<1$ (but not the metric) is equivalent to the $C^0$ topology, because of the interpolation inequalities \eqref{eq interp} below.}, $0\leq \alpha<1$. By the Taylor formula, if $\tau_1,\tau_2\in\cB_{1/2}$ then $\|L_{\tau_1}f-L_{\tau_2} f\|_{L^2}\lesssim \|\tau_1-\tau_2\|_{\infty}\|f'\|_{L^2}$. Hence, since $S_J[\cP_J]$ is nonexpansive,
for every $J\in\bZ$ and every \textit{fixed} $f\in H^1(\mathbb{R})$ (Sobolev space) -- in particular for each $f_n$ as above -- the map $ \mathcal{B}_{1/2}\to\ell^2(\cP_J;L^2(\bR))$ given by
$\tau\mapsto S_J[\cP_J](L_\tau f)$ is Lipschitz continuous (cf.\ also \cite{koller}). On the other hand, Theorem \ref{thm mainthm0} provides a lower bound for the blow-up rate of the Lipschitz constant, depending on $\alpha$, when the input data become progressively less regular.

The instability results in Theorem \ref{thm mainthm0} can be heuristically explained as follows. Consider a smooth signal $f$ with unit $L^2$ norm. The deformed signal $L_\tau f$ has a certain low-frequency mass, but a relatively small energy bump in a quite far dyadic frequency band may occur even if $\|\tau'\|_{L^\infty}\le 1/2$. The latter will propagate along different scattering paths, thus preventing the quantity $\|S_J[\cP_J](L_\tau f)-S_J[\cP_J](f)\|$ from being too small -- assuming that $S_J[\cP_J]$ preserves the norm, which is a consequence of the assumptions in Theorem \ref{thm mainthm0}. A simple scaling argument shows that the same phenomenon can happen even when $\|\tau\|_{C^\alpha}\to 0$, $0\leq \alpha<1$, along with a corresponding loss of regularity for $f$. 

To summarize, Theorem \ref{thm mainthm0} and Proposition \ref{pro pro0} below show that, as far as the Lipschitz continuity under $C^\alpha$ deformation is concerned, the threshold $\alpha=1$ is critical, both for functions in $L^2$ and for functions with finite scattering norm. On the other hand, we have the positive result \eqref{eq mallat0} in the case $\alpha=2$. The following stability result essentially fills this gap -- we assume the same condition on the filters as in \cite{mall cpam} (see Section \ref{sec a brief review}). 
\begin{theorem}\label{thm mainthm1}
Consider $0<\alpha<1$. There exists a constant $C>0$ such that, for all $J \in \bZ$, $f \in L^2(\rd)$ with $\| U[\cP_J]f\|_1 < \infty$, and $\tau \in C^{1+\alpha}(\rd;\rd)$, with $\|D\tau\|_{L^\infty}\le 1/2$, 
\begin{equation}\label{eq main fm1}
        \| S_J[\cP_J](L_\tau f) - S_J[\cP_J](f) \| \le  C K_{1+\alpha}(\tau) \|U[\cP_J]f\|_1,
\end{equation} with
\[ K_{1+\alpha}(\tau)= 2^{-J}\|\tau\|_{L^\infty} + \badmax \|D\tau\|_{L^\infty} + |D\tau|_{\Ca} .\]
\end{theorem}
The definition of the $C^\alpha$ seminorm $|\cdot|_{C^\alpha}$ is recalled in Section \ref{sec notation}.  
This result arises as a refinement of \cite[Theorem 2.12]{mall cpam}, with which it shares the backbone structure of the proof. A careful inspection of the latter suggests that lower levels of deformation complexity (such as logarithmic H\"older regularity) could still give rise to stability results. A substantial rearrangement of some parts of the proof strategy is expected to accommodate even lower regularity levels, such as Dini continuous deformations. In any case, we preferred to keep the technicalities at a minimum and to use the more natural $C^\alpha$ scale, also in view of applications. 

The combination of the previous results provides us with a substantially complete picture on the interplay between stability and deformation regularity. Notably, the case of Lipschitz (or even  $C^1(\rd;\rd)$) distortions remains open. A dimensional argument shows that, for $f\in L^2(\rd)$, the expected bound would have the form
\begin{equation}\label{eq c1 ideal}
 \| S_J[\cP_J](L_\tau f) - S_J[\cP_J](f) \| \le  C \big( 2^{-J}\|\tau\|_{L^\infty} +\|D\tau\|_{L^\infty}\big)\|f\|_{L^2}.
\end{equation}
While proving this estimate is definitely an ambitious goal, this problem seems to be out of reach at the current time. Interestingly, we are able to show that it holds up to arbitrarily small losses, at least in dimension $1$. As customary in harmonic analysis, to accomplish this goal we consider the case of band-limited functions $f$, with $\widehat{f}$ supported in the frequency ball $|\omega|\leq R$, say, and determine the blow-up rate in the above regime as $R\to+\infty$. The following stability result for Lipschitz deformations shows that such a rate is indeed smaller than $R^\varepsilon$ for every $\varepsilon>0$.

First, we assume that there are $C,\beta>0$ such that
\begin{equation}\label{eq emb int} 
\|U[\cP_0]f\|_1\leq C\log^\beta(e+R)\|f\|_{L^2}
\end{equation}
for every $f\in L^2(\rd)$ with $\widehat{f}(\omega)$ supported in the ball $|\omega|\leq R$. Such an estimate holds in dimension $d=1$, for every $\beta>1$, as a consequence of Proposition \ref{pro embed} below under an admissibility condition on the filters detailed in \cite{Waldspurger}. There is reason to believe that such a logarithmic bound holds in arbitrary dimension (cf.\ for instance \cite[Lemma 2.8]{mall cpam} and the related remarks). 

\begin{theorem}\label{thm mainthm 3} Assume \eqref{eq emb int}. For every $\varepsilon>0$ there exists $C>0$ such that, for every $\tau \colon \rd\to\rd$ bounded and globally Lipschitz, with $\|D\tau\|_{L^\infty}\leq 1/2$, and every $f\in L^2(\rd)$ with $\widehat{f}(\omega)$ supported in the ball $|\omega|\leq R$, $R>0$, and every $J\in\bZ$, we have 
\begin{equation}\label{eq ideal esp} 
\| S_J[\cP_J](L_\tau f) - S_J[\cP_J](f) \| \le C \big( \log^\beta(e+2^J R)2^{-J}\|\tau\|_{L^\infty} +(1+2^J R)^\varepsilon\|D\tau\|_{L^\infty}\big)\|f\|_{L^2}
 .
\end{equation}
\end{theorem}
The proof is based on a nonlinear interpolation argument in the setting of Besov spaces, which is in turn a refinement of a classical technique that has already been successfully developed in the literature in connection with nonlinear estimates for PDEs \cite{lions,maligranda,peetre,tartar}.


To conclude, we observe that it would be also very interesting to investigate similar stability issues for scattering-type transforms associated with other semi-discrete frames, such as curvelet or shearlet systems \cite{candes,guo}, in view of their prominent role in image processing. Also, from a mathematical perspective, it is natural to wonder whether the above results are robust enough to encompass more general operators than $L_\tau$, Fourier integral operators being the natural candidates \cite{candesfio}. We postpone the study of these problems, that require novel ideas and techniques, to future works.


\section{Preliminaries and review of the scattering transform}\label{sec notation}

\subsection{Notation} 
{\setlength{\parindent}{0cm} 
\setlength{\parskip}{0.2cm}
The open ball of $\rd$ centered at $x_0$ with radius $r>0$ is denoted by $B_r(x_0)$. For a differentiable map $\tau \colon \rd\to\rd$, we denote by $D\tau(x)$ its derivative as a linear map $\rd\to\rd$, hence we write $|D\tau(x)|$ for the operator norm of this map and also set $\|D\tau\|_{L^\infty}=\| |D\tau| \|_{L^\infty}$. Similarly, for a scalar-valued function $f$, $\|\nabla f\|_{L^\infty}=\| |\nabla f| \|_{L^\infty}$.

The Fourier transform of $f$ 
is normalized here as
	\[ \widehat{f}(\omega) =\mathcal{F}(f)(\omega)= \ird e^{-i\omega \cdot x} f(x)\, dx. \]


Given an index set $\Omega$ and a collection of operators $T[p] \colon L^2(\rd)\to L^2(\rd)$  indexed by $p\in\Omega$, we set \[ T[\Omega] = \{T[p]\}_{p \in \Omega}.\] Unless otherwise stated, the standard norm in this context is that of $\ell^2(\Omega;L^2(\rd))$, namely
\[ \| T[\Omega]f \|^2 = \sum_{p \in \Omega} \| T[p]f \|^2_{L^2}, \quad f \in L^2(\rd). \] 

In the proofs, for brevity, we will heavily make use of the symbol $A \lesssim B$, meaning that the underlying inequality holds up to a positive constant factor, namely 
\[ A \lesssim B \quad\Longrightarrow\quad\exists\, C>0\,:\,A \le C B. \] If the constant $C=C(\nu)$ depends on some parameter $\nu$ we write $A\lesssim_\nu B$. Moreover, $A \approx B$ means that $A$ and $B$ are \textit{equivalent quantities}, that is both $A \lesssim B$ and $B\lesssim A$ hold. 

In the rest of the note, all the derivatives are to be understood in the distribution sense, unless otherwise noted. }

\subsection{Relevant function spaces}	
Consider an open subset $A\subseteq \rd$ and set $Y=\bR^n$ or $Y=\bC$.
Given a nonnegative integer $k$ we introduce the space $C^k(A;Y)$ of all the continuously differentiable functions $f \colon A\to Y$ with bounded derivatives up to order $k$, with the natural norm $\|f\|_{C^{k}(A)}\coloneqq \max_{|\beta|\leq k}\sup_{x\in A} |\partial^\beta f(x)|$. 

We define the $\alpha$-H\"older seminorm, $0<\alpha < 1$, and the Lipschitz seminorm of $f \colon A\to Y$ by 
    \[ | f|_{C^\alpha(A)} \coloneqq \sup_{\substack{x,y \in A \\ x\ne y}} \frac{|f(x)-f(y)|}{|x-y|^\alpha},
\qquad | f|_{\Lip(A)} \coloneqq \sup_{\substack{x,y \in A \\ x\ne y}} \frac{|f(x)-f(y)|}{|x-y|}. 
\]
The space $C^\alpha(A;Y)$, $\alpha>0$, consists of all the functions $f \colon A\to Y$, continuously differentiable up to the order $[\alpha]$ (integer part of $\alpha$), such that 
\[
\|f\|_{C^\alpha(A)}:=\|f\|_{C^{[\alpha]}(A)}+ 
\sum_{|\gamma|=[\alpha]}|\partial^\gamma f|_{C^{\alpha-[\alpha]}(A)}<\infty.
\]
When there is no risk of confusion we usually omit the codomain $Y$ and also the domain in the case where $A=\rd$, writing for instance $C^\alpha$ in place of $C^\alpha(\rd;Y)$ for simplicity. We also recall the elementary interpolation inequality 
\begin{equation}\label{eq interp}
|f|_{C^\alpha}\leq 2^{1-\alpha} \|f\|^{1-\alpha}_{L^\infty}\|\nabla f\|_{L^\infty}^{\alpha}. 
\end{equation}

We collect here some basic properties that will be used below.  
\begin{proposition} Fix $A\subset\rd$ and $0<\alpha<1$.

\begin{itemize}[leftmargin=*]
\item (Fractional Leibniz rule) If $f,g \in C^{\alpha}(A;\bC)$, then
\begin{equation}\label{eq frac leib}
    | f g |_{\Ca(A)} \le | f |_{\Ca(A)} \|g \|_{L^\infty(A)} + \| f\|_{L^\infty(A)} |g |_{\Ca(A)}. 
\end{equation}

\item (Schauder estimates) Assume that $F\colon \bR^n \to \bC$ is Lipschitz. For $h \in \Ca(A;\bR^n)$,
\begin{equation}\label{eq schaud}
    | F(h) |_{\Ca(A)} \le | F|_{\Lip(h(A))} | h|_{\Ca(A)}. 
\end{equation}
\end{itemize}

\end{proposition}

\begin{proof}
The fractional Leibniz rule is readily obtained by noting that for all $x,y \in \rd$, $x\ne y$,
\[ \frac{|f(x)g(x)-f(y)g(y)|}{|x-y|^\alpha} \le \frac{|f(x)-f(y)|}{|x-y|^\alpha}|g(x)| + |f(y)|\frac{|g(x)-g(y)|}{|x-y|^\alpha}.\]
The inequality in \eqref{eq schaud} follows similarly -- as long as $h(x) \ne h(y)$,
\[ \frac{|F(h(x))-F(h(y))|}{|x-y|^\alpha} \le \frac{|F(h(x))-F(h(y))|}{|h(x)-h(y)|} \frac{|h(x)-h(y)|}{|x-y|^\alpha}.\] 
\end{proof}


\subsection{A brief review of the wavelet scattering transform} \label{sec a brief review} In this section we gather some basic facts and results concerning the mathematical analysis of the scattering transform, mainly in order to fix the notation. More details can be found in \cite{brunamallat, mall cpam}.

The basic ingredient is a complex wavelet $\psi \in L^1(\rd) \cap L^2(\rd)$ with at least one vanishing moment ($\widehat{\psi}(0)=0$), satisfying appropriate conditions that are stated below. 
    
Let $G$ be a finite subgroup of rotations in $\rd$, also comprising the reflection operator $-I$. For every $\lambda = 2^j r$ with $j \in \bZ$ and $r \in G$ we set
\[ \psi_\lambda(x) \coloneqq 2^{jd} \psi(2^j r^{-1} x).
\]  The frequency filtering corresponding to $\psi_\lambda$ is thus obtained by convolution, namely for $f \in L^2(\rd)$ we set
\[ W[\lambda]f(x) \coloneqq f*\psi_\lambda (x) = \ird f(y) \psi_\lambda(x-y) dy, \quad x \in \rd.\]
If both $f$ and $\widehat{\psi}$ are real functions it is easy to realize that $W[-\lambda]f = \overline{W[\lambda]f}$. We thus conveniently consider the quotient $G^+ = G/\{\pm I\}$, so that all the pairs of rotations $r$ and $-r$ are identified.

\begin{remark}

As a concrete reference model one typically considers a Gabor-like wavelet $\psi$ such as 
\[
\psi(x) = e^{i \eta \cdot x}\theta(x), \quad x \in \rd, 
\] for some $\eta \in \rd$ and a function $\theta$ with real-valued Fourier transform $\widehat{\theta}$ essentially supported in a low-frequency ball centered at the origin with radius of the order of $\pi$. Then $\wh{\psi_\lambda}(\omega) = \widehat{\psi}(2^{-j}r^{-1}\omega)=\widehat{\theta}(\lambda^{-1}\omega-\eta)$ is concentrated in a ball centered at $\lambda \eta$ of approximate size $|\lambda|\coloneqq 2^j$.

\end{remark}
Given $J \in \bZ$ we introduce the index set 
\begin{equation}\label{eq def lambdaj}
\Lambda_J \coloneqq \{ \lambda=2^j r : j > -J, \, r \in G^+\}. \end{equation} 
The filter bank $W[\Lambda_J] \coloneqq \{W[\lambda]\}_{ \lambda \in \Lambda_J}$ is thus not able to detect a low-frequency component of a real signal $f$ corresponding to a region of the frequency space with size of the order of $2^{-J}$. Nevertheless, the latter can be captured by a suitable average $A_J$ with a dilated low-pass filter $\phi \in L^1(\rd)\cap L^2(\rd)$ such that $\phi$ is a non-negative real-valued function with $\widehat{\phi}(0)=1$ (having in mind a Gaussian function as a model), that is
\[ A_J f \coloneqq f * \phi_{2^J}, \qquad \phi_{2^J}(x) \coloneqq 2^{-Jd}\phi(2^{-J}x), \quad x \in \rd. \] To ensure that the frequency content of $f$ is fully preserved by a wavelet analysis at a scale $2^J$ it is enough that the supports of the filters obtained by dilations of $\phi$ and $\psi$ cover the whole frequency space. To be more precise, let $W_J f$ be the wavelet analysis of $f$ associated with $\phi$ and $\psi$, namely the collection of signal components indexed by $\{J\}\cup \Lambda_J$ given by
\[ W_J f \coloneqq \{ A_J f, W[\Lambda_J]f\}. \] It is not difficult to show that $W_J$ is an isometry from $L^2(\rd;\bR)$ (real-valued functions) to $\ell^2(\{J\}\cup \Lambda_J;L^2(\rd))$ if and only if the following Littlewood-Paley condition holds for almost every $\omega \in \rd$:
\begin{equation}\label{eq LP}
    |\widehat{\phi}(2^J \omega)|^2 + \frac{1}{2} \sum_{\lambda \in \Lambda_J}\left[ |\widehat{\psi}(\lambda^{-1}\omega)|^2 +|\widehat{\psi}(-\lambda^{-1}\omega)|^2\right]= 1.
\end{equation}  Hence
\[ \| f \|^2_{L^2} = \| W_J f \|^2_{\ell^2 L^2} \coloneqq \| A_J f \|_{L^2}^2 + \sum_{\lambda \in \Lambda_J} \|W[\lambda]f\|_{L^2}^2
\] 
(see \cite[Proposition 2.1]{mall cpam} for further details).
Slight modifications are needed in the case of complex-valued signals $f$ in order to accommodate all the rotations. The wavelet analysis $W_Jf$ is then accordingly defined including $W[-\Lambda_J]f$ as well, and unitarity of $W_J f$ is ensured by the condition 
\[ |\widehat{\phi}(2^J \omega)|^2 + \sum_{\lambda \in \Lambda_J} \left[ |\widehat{\psi}(\lambda^{-1}\omega)|^2+ |\widehat{\psi}(-\lambda^{-1}\omega)|^2 \right]= 1. \] We simply discuss below the case where $f \in L^2(\rd)$ takes real values to lighten the presentation. 

We say that $\phi$ and $\psi$ are \textit{scattering filters} if:
\begin{itemize}
    \item Given $J \in \bZ$ and a group of rotations $G$ as above, the condition in \eqref{eq LP} is satisfied.
    \item $\widehat{\psi}$ is real-valued, and  $\widehat{\phi}$ is real-valued and symmetric. Moreover, $\phi$ is non-negative and $\hat{\phi}(0)=1$. 
    \item Both $\phi(x)$ and $\psi(x)$ are twice differentiable and decay like $O((1+|x|)^{-d-3})$ together with their first and second partial derivatives\footnote{In \cite{mall cpam} it is assumed a decay condition $O((1+|x|)^{-d-2})$ instead. Nevertheless, it seems that even in that case the decay of order $-(d+3)$ is needed, for instance in order to make the integral in \cite[(E.26)]{mall cpam} convergent or to suitably bound the last terms in \cite[(E.30)]{mall cpam} in such a way to obtain \cite[(E.31)]{mall cpam}.}.
\end{itemize}
We now introduce the general index set $\Lambda_{\infty} \coloneqq 2^\bZ\times  G^+$ and the corresponding space $\cP_\infty \coloneqq \bigcup_{m\ge0} \Lambda_{\infty}^m$ of all the possible finite paths, where it is understood that $\Lambda_\infty^0=\{\emptyset\}$. The one-step scattering propagator $U[\lambda]$, $\lambda \in \Lambda_\infty$, coincides with a modulus wavelet localization: $U[\lambda]f \coloneqq |W[\lambda]f| = |f*\psi_\lambda|$. More generally, the path-ordered scattering propagator $U\colon \cP_\infty \times L^2(\rd) \to L^2(\rd)$ acts along a path $p=(\lambda_1, \ldots, \lambda_m) \in \Lambda_\infty^m$ of length $m\ge 1$ by
\[ U[p]f \coloneqq U[\lambda_m] \cdots U[\lambda_2]U[\lambda_1]f. \]
For the empty path $p=\emptyset$ we set $U[\emptyset]f=f$. 

The collection of all the paths with finite length and components in $\Lambda_J$ is $\cP_J \coloneqq \bigcup_{m=0}^\infty \Lambda_J^m$ (again $\Lambda_J^0 = \{\emptyset\}$). The windowed scattering transform $S_J[\cP_J]$ is then defined as follows:
\[ S_J[\cP_J] f = \{S_J[p] f\}_{ p \in \cP_J}, \quad S_J[p]f \coloneqq A_J U[p] f. \]

\begin{figure}[h!]
\begin{forest}
	treenode/.style = {align=center, inner sep=0pt, text centered, font=\sffamily},
	for tree={
	grow'=north,
	delay={
		edge label/.wrap value={node[midway, font=\sffamily\scriptsize, above]{#1}},
		},
			}
	[$f$, name=L0,
	[{$U [\lambda_1 ] f$}, name=L1sx ,
	[{$\scalemath{0.8}{U [\lambda_1,\lambda_2] f}$}, name=L2sx, [$\vdots$] [$\vdots$] [$\vdots$] [$\vdots$]
	]
	[, tikz={\draw [fill=white] (.anchor) circle[radius=2pt];}  [$\vdots$] [$\vdots$] 
	]
	[, tikz={\draw [fill=white] (.anchor) circle[radius=2pt];} [$\vdots$] [$\vdots$]
	]
	]
	[,tikz={\draw [fill=white] (.anchor) circle[radius=2pt];} 
	[, tikz={\draw [fill=white] (.anchor) circle[radius=2pt];} [$\vdots$]], [, tikz={\draw [fill=white] (.anchor) circle[radius=2pt];} [$\vdots$]]]
	[{$U [\lambda_1 ] f$}, name=L1dx,
	[, tikz={\draw [fill=white] (.anchor) circle[radius=2pt];} [$\vdots$] 
	]
	[,tikz={\draw [fill=white] (.anchor) circle[radius=2pt];} [$\vdots$]
	]
	[$\scalemath{0.8}{U [\lambda_1,\lambda_2] f}$, name=L2dx,
	[,tikz={\draw [fill=white] (.anchor) circle[radius=2pt];}  [$\vdots$][$\vdots$][$\vdots$]
	]
	[$\cdots$,
	[$\scalemath{0.8}{U [\lambda_1,\lambda_2,\ldots,\lambda_m] f}$, name=Ldxend,  [$\vdots$,for ancestors={edge={red,thick}}] [$\vdots$] [$\vdots$]]
	]
	]
	]
	]
	]
	\draw [blue, dashed, ->, shorten >=3pt] (L0) -- ++(-27pt,-20pt) node [circle,fill=blue,inner sep=1pt,minimum size=3pt,label={[xshift=0pt, yshift=-15pt]  $\scalemath{0.7}{f*\phi_{2^J}}$}] {};;
	\draw [blue, dashed, dashed, ->, shorten >=3pt] (L1sx) -- ++(-40pt,-30pt) node [circle,fill=blue,inner sep=1pt,minimum size=3pt,label={[xshift=0pt,yshift=-20pt] $\scalemath{0.7}{(U [\lambda_1 ] f )*\phi_{2^J}}$}] {};;
	\draw [blue, dashed, ->, shorten >=3pt] (L2sx) -- ++(-40pt,-30pt) node [circle,fill=blue,inner sep=1pt,minimum size=3pt,label={[xshift=0pt,yshift=-20pt] $\scalemath{0.7}{(U [\lambda_1, \lambda_2] f )*\phi_{2^J}}$}] {};;
	\draw [blue, dashed, ->, shorten >=3pt] (L1dx) -- ++(50pt,-25pt) node [circle,fill=blue,inner sep=1pt,minimum size=3pt,label={[xshift=0pt,yshift=-20pt] $\scalemath{0.7}{(U [\lambda_1 ] f )*\phi_{2^J}}$}] {};;		\draw [blue, dashed, ->, shorten >=3pt] (L2dx) -- ++(50pt,-25pt) node [circle,fill=blue,inner sep=1pt,minimum size=3pt,label={[xshift=0pt,yshift=-20pt] $\scalemath{0.6}{(U [\lambda_1,\lambda_2] f)*\phi_{2^J}}$}] {};;			
	\draw [blue, dashed, ->, shorten >=3pt] (Ldxend) -- ++(50pt,-25pt) node [circle,fill=blue,inner sep=1pt,minimum size=3pt,label={[xshift=0pt,yshift=-20pt] $\scalemath{0.6}{(U [\lambda_1,\lambda_2,\ldots,\lambda_m] f)*\phi_{2^J}}$}] {};;			
			\end{forest}
		\caption{The scattering network architecture, as described above. The index $\lambda_l \in \Lambda_J$ corresponds to the $l$-th layer. In blue: some features. In red: an example of a path $q=(\lambda_1,\lambda_2\ldots,\lambda_m) \in \Lambda_J^m$ of length $m$.}
\end{figure}
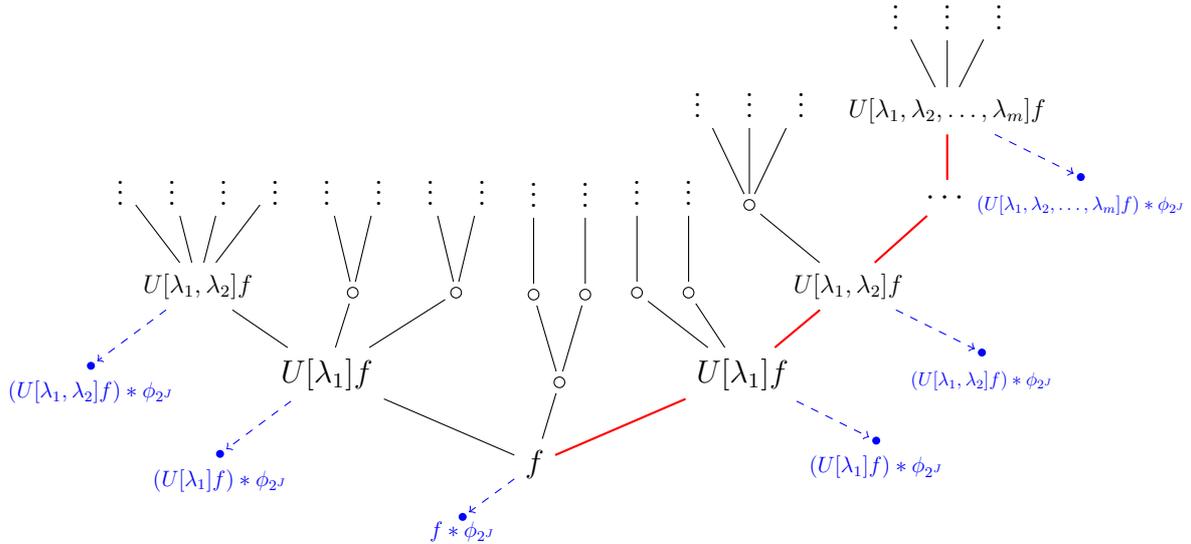

The assumptions satisfied by the underlying scattering wavelets allow one to show that $S_J[\cP_J]$ has the desired properties from a feature map with values in $\cH=\ell^2(\cP_J;L^2(\rd))$ as discussed in the introduction.

\noindent \textit{Lipschitz regularity}. It is proved in \cite[Proposition 2.5]{mall cpam} that $S_J[\cP_J] \colon L^2(\rd) \to \cH$ is a nonexpansive transform, namely
\[ \adm{S_J[\cP_J]f - S_J[\cP_J]h} \le \|f-h\|_{L^2}, \quad f,h\in L^2(\rd).\]

\noindent \textit{Norm preservation}. Provided that the filters satisfy additional admissibility conditions (see \cite[Theorem 2.6]{mall cpam} or \cite[Theorem 3.1]{Waldspurger} in dimension $d=1$), $S_J[\cP_J]$ preserves the norm of the input signal:
\[ \| f \|_{L^2}^2 = \|S_J[\cP_J]f\|^2 = \sum_{p \in \cP_J} \|S_J[p]f\|_{L^2}^2, \quad f \in L^2(\rd).\] 

\noindent \textit{Translation invariance}. It is proved in \cite[Proposition 2.9]{mall cpam} that the scattering distance $\adm{S_J[\cP_J]f-S_J[\cP_J]h}$ is nonincreasing when $J$ increases, and
the scattering metric is asymptotically translation invariant, as proved in \cite[Theorem 2.10]{mall cpam}:\[ \lim_{J \to +\infty} \adm{S_J[\cP_J](T_xf)-S_J[\cP_J](f)} = 0, \quad \forall x \in \rd, \, f \in L^2(\rd). \]

\noindent \textit{Stability to small deformations}. As already anticipated in the Introduction, the stability bound \eqref{eq mallat0} is proved in \cite[Theorem 2.12]{mall cpam} for functions $f$ such that
\[ \|U[\cP_J]f \|_1 = \sum_{m\ge 0} \| U[\Lambda_J^m]f \|<\infty.
\]

Let us discuss some additional properties of the scattering transform which are used below. 

\noindent \textit{Covariance properties}. The joint action of scaling and rotation by $2^l g \in 2^\bZ \times G$ on a signal $f$ is given by $(2^l g \circ f)(x) \coloneqq  f(2^l g x)$, $x \in \rd$, while for a path $p = (\lambda_1, \ldots, \lambda_m) \in \cP_\infty$ of length $m$ we set $2^lgp \coloneqq (2^lg \lambda_1, \ldots, 2^lg \lambda_m)$. It is not difficult to show that the one-step propagator is somehow covariant to scaling and rotations, namely $U[\lambda](2^lg \circ f) = 2^l g \circ U[2^{-l}g\lambda]f$, $\lambda\in\Lambda_\infty$. In view of the cascading structure of the scattering transform, this property reflects into
\begin{equation}\label{eq mall 2.18}
    U[p](2^lg \circ f) = 2^l g \circ U[2^{-l}gp]f, \quad p \in \cP_\infty, 
\end{equation} and
\begin{equation}\label{eq cov scatt}
    S_J[p](2^lg \circ f) = 2^l g \circ S_{J+l}[2^{-l}gp]f, \quad p \in \cP_J. 
\end{equation}

\noindent \textit{Additivity on separated signals}. The following simple result shows that $S_J[\cP_J]$ is additive on functions that are separated in the wavelet domain. 
\begin{lemma}\label{lemma sep}
Let $f,g\in L^2(\rd)$ be such that, for every $\lambda\in\Lambda_J$,
\[
f\ast \psi_\lambda=0\quad{\it or}\quad  g\ast \psi_\lambda=0.
\]
Then 
\[
S_J[\cP_J](f+g)=S_J[\cP_J](f)+S_J[\cP_J](g). 
\]
\end{lemma}
\begin{proof}
Since the convolution with $\phi_{2^J}$ is a linear operator, it suffices to prove that 
\[
U[p](f+g)=U[p]f+U[p]g, \quad p \in \cP_J.
\]
Consider then $p=(\lambda_1,\ldots,\lambda_m)\in\Lambda_J^m$. Since 
\[
U[\lambda_1,\ldots,\lambda_m]=U[\lambda_2,\ldots,\lambda_{m}]U[\lambda_1],
\]  
it is enough to show that, for every $\lambda\in\Lambda_J$,
$U[\lambda](f+g)=U[\lambda]f+U[\lambda]g$ \textit{and} one of the two terms on the right-hand side vanishes. The claim follows at once from the assumption and the definition of the one-step propagator $U[\lambda]f=|f\ast\psi_\lambda|$.
\end{proof}
We finally present the following embeddings in dimension $d=1$, obtained by means of the scattering decay results proved in \cite{Waldspurger}. The formula \eqref{eq embell2} is essentially known, cf.\ \cite[Lemma 2.8]{mall cpam} -- the latter was proved in arbitrary dimension under a more restrictive admissibility condition on the wavelet $\psi$ (see \cite[Theorem 2.6]{mall cpam}). The estimate \eqref{eq embell1} seems new. 

\begin{proposition}\label{pro embed}
Let $\psi \in L^1(\bR) \cap L^2(\bR)$ satisfy the following Littlewood-Paley inequality, for every $\omega\in\bR$: \[
\frac{1}{2}\sum_{j\in\bZ} \left[ |\widehat{\psi}(2^{-j}\omega)|^2+ |\widehat{\psi}(-2^{-j}\omega)|^2 \right]\leq  1.
\]

\noindent Moreover, assume that 
\[
|\widehat{\psi}(-2^{-j}\omega)|\leq |\widehat{\psi}(2^{-j}\omega)|
\]
for $\omega>0$ and $j\in\bZ$, provided that for every $\omega$ the condition holds with strict inequality for at least one value of $j$.

\noindent Finally, assume that $|\widehat{\psi}(\omega)|=O(|\omega|^{1+\varepsilon})$ for some $\varepsilon>0$, as $\omega\to0$. 

There exists $C>0$ such that, for every $J\in\bZ$,
\begin{equation}\label{eq embell2}
\|U[\cP_J]f\|^2\leq C\int_{\bR}|\widehat{f}(\omega)|^2\log(e+2^J|\omega|)\, d\omega.
\end{equation}
Moreover, for every $\beta>2$ there exists $C>0$ such that, for every $J\in\bZ$,
\begin{equation}\label{eq embell1}
\|U[\cP_J]f\|_1
\leq  C\Big(\int_{\bR}|\widehat{f}(\omega)|^2\log^\beta(e+2^J|\omega|)\, d\omega\Big)^{1/2}.
\end{equation}
\end{proposition}

\begin{proof}[Proof of Proposition \ref{pro embed}]
First of all, we note that it suffices to prove the estimates \eqref{eq embell2} and \eqref{eq embell1} for $J=0$. The claim then follows by a scaling argument. More precisely, consider $f_J(x)\coloneqq 2^{J/2}f(2^{J}x)$ and note that, by \eqref{eq mall 2.18}, we have $\|U[\Lambda_0^m]f_J\|=\|U[\Lambda_J^m]f\|$ and therefore $\|U[\cP_0]f_J\|=\|U[\cP_J]f\|$ as well. 

By \cite[Theorem 3.1]{Waldspurger}, the assumptions in the statement imply that, for $m\geq 2$,
\begin{equation}\label{eq wald}
    \|U[\Lambda^m_0]f\|^2\leq \frac{1}{2\pi}\int_{\bR}|\widehat{f}(\omega)|^2 A_m(\omega)\,d\omega,
\end{equation}
where\footnote{The factor $1/(2\pi)$ in \eqref{eq wald} does not appear in \cite[Theorem 3.1]{Waldspurger} because of a different normalization of the Fourier transform. The constant $r$ below will be different from that in \cite[Theorem 3.1]{Waldspurger} as well, for the same reason.} $A_m(\omega)=1-e^{-2\omega^2/(r a^{m})^2}$, for suitable $a>1$, $r>0$. 

Since $\|U[\cP_0]f\|^2=\sum_{m=0}^\infty \|U[\Lambda_0^m]f\|^2$, and $ \|U[\emptyset]f\|^2+\|U[\Lambda_0]f\|^2\leq 2\|f\|^2_{L^2}$ ($U[\Lambda_0]$ being nonexpansive), in order to obtain \eqref{eq embell2} it is enough to verify that 
\begin{equation}\label{eq somma}
\sum_{m=2}^\infty A_m(\omega)\leq C \log(e+|\omega|/r) 
\end{equation}
for some constant $C>0$, possibly depending on $a$ only. 

A straightforward change of variable shows that we can suppose $r=1$ without loss of generality. The estimate \eqref{eq somma} is satisfied if $|\omega|\leq a^2$ because $A_m(\omega)\lesssim (\omega/a^m)^2$. On the other hand, if $|\omega|\geq a^2$ we conveniently split the sum in \eqref{eq somma} in two parts accounting for $m\leq N$ and $m>N$, where $N\geq 2$ is such that $a^N\leq |\omega|<a^{N+1}$. Using that $A_m(\omega)\leq 1$ and $A_m(\omega)\lesssim (\omega/a^m)^2\lesssim a^{2(N-m)}$ in the two regimes, respectively, we obtain 
\[
\sum_{m=2}^\infty A_m(\omega)\lesssim \sum_{2\leq m\leq N} 1+\sum_{m>N} a^{2(N-m)}\lesssim N\leq \log_a|\omega|,
\]
which gives \eqref{eq somma}.

Let us now prove \eqref{eq embell1} with $J=0$ in light of the previous arguments. By \eqref{eq wald} we see that it is sufficient to prove the bound
\begin{equation}\label{eq somma2}
\sum_{m=2}^\infty \Big(\int_{\bR}|\widehat{f}(\omega)|^2 A_m(\omega)\,d\omega\Big)^{1/2}\lesssim \Big(\int_{\bR}|\widehat{f}(\omega)|^2\log^\beta(e+|\omega|)\, d\omega\Big)^{1/2}.
\end{equation}
Since $\beta>2$, the latter will follow from the pointwise bound 
\[
\frac{A_m(\omega)}{\log^\beta(e+|\omega|)}\lesssim \frac{1}{m^\beta}.
\]
This estimate clearly holds if $|\omega|\geq a^{m/2}$, since $A_m(\omega)\leq 1$. If $|\omega|< a^{m/2}$ we have
\[
\frac{A_m(\omega)}{\log^\beta(e+|\omega|)}\leq A_m(\omega) \lesssim
\frac{\omega^2}{a^{2m}}\leq  \frac{1}{a^{m}}\lesssim \frac{1}{m^\beta}.
\]
\end{proof}
\begin{remark}
It is worthwhile to point out that \eqref{eq somma2} does not hold for $\beta=2$, as evidenced by the following example.  Consider a function $f\in L^2(\bR)$ whose Fourier transform is supported in $[ra^2,+\infty)=\cup_{k\geq 2}\Omega_k$, with $\Omega_k=[ra^k,ra^{k+1})$, and takes a constant value on each $\Omega_k$, adjusted so that  $\int_{\Omega_k} |\widehat{f}(\omega)|^2\log^2(e+|\omega|)\, d\omega=1/(k\log^2 k)$. We then have \[\int_{\bR} |\widehat{f}(\omega)|^2\log^2(e+|\omega|)\, d\omega<\infty.\] On the other hand, if $k\geq m$, on $\Omega_k$ we have $A_m(\omega)/\log^2(e+|\omega|)\gtrsim 1/\log^2(e+|\omega|)\gtrsim 1/k^2$, so that 
\begin{align*}
\sum_{m=2}^\infty \Big(\int_{\bR}|\widehat{f}(\omega)|^2 A_m(\omega)\,d\omega\Big)^{1/2}&\geq  \sum_{m=2}^\infty\Big(\sum_{k\geq m} \int_{\Omega_k} |\widehat{f}(\omega)|^2A_m(\omega)\, d\omega\Big)^{1/2}\\
&\gtrsim \sum_{m=2}^\infty\Big(\sum_{k\geq m} \frac{1}{k^3\log^2 k}\Big)^{1/2}=\infty.
\end{align*}
Indeed, the latter series is readily seen to diverge, since
\[
\sum_{k\geq m} \frac{1}{k^3\log^2 k}\geq \int_m^{+\infty}\frac{1}{x^3 \log^2 x}\,dx\gtrsim \frac{1}{m^2\log^2 m},
\]
where we resorted to integration by parts in the last step. 
\end{remark}

\section{Instability results for $C^\alpha$ regularity, $0\leq \alpha<1$} 
This section is devoted to the instability phenomenon occurring for deformations with regularity $C^\alpha$, $0\leq \alpha<1$, already illustrated in the Introduction. We begin with the proof of Theorem \ref{thm mainthm0}. 
\begin{proof}[Proof of Theorem \ref{thm mainthm0}] Let us begin with the proof of \eqref{eq cont 0}. First of all, under the stated assumptions on $\phi,\psi$ we infer from \cite[Theorem 3.1]{Waldspurger}\footnote{The assumptions of Theorem \cite[Theorem 3.1]{Waldspurger} are the same as those of Proposition \ref{pro embed} and are therefore satisfied here. Indeed, if \eqref{eq LP} holds for the specified $J$, by rescaling one sees that it holds for every $J\in\bZ$. Letting $J\to-\infty$, since $\widehat{\phi}$ tends to $0$ at infinity we see that 
\[
\frac{1}{2}\sum_{j\in\bZ} \left[ |\widehat{\psi}(2^{-j}\omega)|^2+ |\widehat{\psi}(-2^{-j}\omega)|^2 \right]=1.
\]
} that 
\[
\lim_{m\to\infty}\|U[\Lambda^m_J]f\|=0
\]
for every $f\in L^2(\bR)$. In view of the Littlewood-Paley condition \eqref{eq LP}, the latter result implies that $S_J[\cP_J]$ preserves the norm (cf.\ the proof of \cite[Theorem 2.6]{mall cpam}).

Consider a compactly supported $f\in C^\infty(\bR;\bR)$ with $f(x)=x$ for $0 \le x \le 2\pi$. Let $\varphi\in C^\infty(\bR;\bR)$, supported in $[0,2\pi]$, be such that $0<\varphi(x)\leq 1$ for $0<x<2\pi$. 

Consider the deformation function defined by
\[
\tau(x)=-\frac{A}{N} \sin(N x)\varphi(x),
\]
where $N\in\mathbb{N}$, $N\geq 1$, will be chosen later (large enough) and $A>0$ is fixed in such a way that $A(1+\|\varphi'\|_{L^\infty})\le 1/2$, hence $\|\tau'\|_{L^\infty}\le 1/2$ for every $N$. 

For future reference, we remark that
\begin{equation}\label{eq cont 4}
x-\tau(x)\in [0,2\pi]\qquad {\rm for}\ x\in [0,2\pi].
\end{equation}
This follows from the fact that, since $\|\tau\|_{L^\infty}\leq AN^{-1}\leq \pi N^{-1}$, if $x$ belongs to one of the $2N$ subintervals of $[0,2\pi]$ where $\tau$ has constant sign then $x-\tau(x)$ belongs either to the same interval or to an adjacent one.   

We finally set $f_n(x)=2^{n/2}f(2^n x)$ and $\tau_n(x)=2^{-n} \tau(2^n x)$, $n\in\mathbb{N}$, as in the statement. 

A convenient facilitation results from the fact that it suffices to prove the desired estimate \eqref{eq cont 0} for $n=0$, with a constant $C_0$ independent of $J\in\bZ$. This can be readily inferred from the scaling property \eqref{eq cov scatt}, yielding
\begin{equation}\label{eq scal}
\adm{S_J[\cP_J] (L_{\tau_n}f_n)- S_J[\cP_J] (f_n)}=\adm{S_{J+n}[\cP_{J+n}](L_{\tau}f)- S_{J+n}[\cP_{J+n}](f)}.
\end{equation}
 
Let us thus set $n=0$ hereafter. We are going to prove that, for $N\in\bN$ large enough,
\begin{equation}\label{eq cont 14}
\adm{S_J[\cP_J](L_\tau f)-S_J[\cP_J](f)}\gtrsim \frac{1}{N}
\end{equation}
where the hidden constants in the symbols $\gtrsim$, $\lesssim$ and  $\approx$ are always independent of $J\in\bZ$ and $N\in\bN$. 

Let us first discuss the strategy. The function $f$ is concentrated in the frequency region where $|\omega|\lesssim 1$, while the function $L_\tau f -f$ will be shown to be concentrated where $|\omega-N|\lesssim 1$ or $|\omega+N|\lesssim 1$. Therefore, if $N$ is sufficiently large we have that $S_J[\cP_J](L_\tau f)$ approximately coincides with $S_J[\cP_J](L_\tau f-f)+S_J[\cP_J](f)$ by virtue of Lemma \ref{lemma sep}, hence
\[
\adm{S_J[\cP_J](L_\tau f)-S_J[\cP_J](f)}\approx \adm{S_J[\cP_J](L_\tau f-f)}\gtrsim \|L_\tau f-f\|_{L^2}. 
\]
Making a rigorous argument out of this clue necessarily comes through suitable bounds for the frequency tails of the functions $f$ and $g\coloneqq L_\tau f -f$. To this aim, let us start by noting that $g =-\tau$ as a consequence of \eqref{eq cont 4} and the fact that $f(x)=x$ for $x\in [0,2\pi]$ by design. Then  
\begin{equation}\label{eq hatg}
\widehat{g}(\omega)= i\frac{A}{2N}(\widehat{\varphi}(\omega-N)-\widehat{\varphi}(\omega+N)). 
\end{equation}
For $j,j'\in\bZ$ we write $f_{\leq 2^j}$, $f_{\geq 2^{j'}}$, $f_{2^j\leq \cdot\leq 2^{j'}}$, for the projections of $f$ on the subspace of $L^2(\bR)$ whose Fourier transform in supported in $|\omega|\leq 2^j$, $|\omega|\geq 2^{j'}$ and $2^j\leq |\omega|\leq 2^{j'}$ respectively, and similarly for the function $g$.

By assumption, $\widehat{\psi}$ is compactly supported in $(0,+\infty)$, hence there exists $k\in\bZ$ such that
\[
{\rm supp}\,\widehat{\psi}\subset [2^{-k},2^k].
\]
As a result, we have that 
\begin{equation}\label{eq cont 1}
{\rm supp}\,\widehat{\psi}_\lambda\subset [2^{j-k},2^{j+k}],\qquad \lambda=2^j\in 2^\bZ.
\end{equation}
Let then $j\in\bZ$ be such that $2^{j+k+1}<N\leq 2^{j+k+2}$. 
By \eqref{eq hatg} we have
\begin{align}\label{eq cont 6} 
\|g_{\leq 2^{j+k}}\|_{L^2}^2&=\frac{1}{2\pi}\int_{[-2^{j+k},2^{j+k}]} |\widehat{g}(\omega)|^2\, d\omega\\
&=\frac{A^2}{8\pi N^2}\int_{[-2^{j+k},2^{j+k}]} |\widehat{\varphi}(\omega-N)-\widehat{\varphi}(\omega+N)|^2\, d\omega\nonumber\\
&\leq \frac{A^2}{4\pi N^2}\int_{[-2^{j+k},2^{j+k}]} (|\widehat{\varphi}(\omega-N)|^2+|\widehat{\varphi}(\omega+N)|^2)\, d\omega\nonumber\\
&\leq \frac{A^2}{4\pi N^2}\int_{\bR\setminus[-(N-2^{j+k}),N-2^{j+k}]} |\widehat{\varphi}(\omega)|^2 d\omega\nonumber\\
&\lesssim_\beta \frac{1}{N^\beta}\nonumber
\end{align}
for every $\beta>0$, because $N-2^{j+k}>N/2$ and $\widehat{\varphi}$ is a rapidly decreasing function. On the other hand,
\begin{equation}\label{eq cont 15}
\|g\|_{L^2}^2=\|\tau\|_{L^2}^2=\frac{ A^2}{ N^2}\int_{[0,2\pi]} \sin^2(Nx) \varphi(x)^2\, dx \approx \frac{1}{N^2}
\end{equation}
for $N$ large enough, because the latter integral converges to $\|\varphi\|^2_{L^2}/2$ as $N\to\infty$ due to the Riemann-Lebesgue lemma.

Therefore, we have obtained, for $N$ large enough,  \begin{equation}\label{eq cont 7}
\|g_{\geq 2^{j+k}}\|_{L^2}^2 = \|g\|_{L^2}^2-\|g_{\leq 2^{j+k}}\|_{L^2}^2 \gtrsim \frac{1}{N^2}. 
\end{equation}
Furthermore, 
\begin{equation}\label{eq cont 8}
\|f_{\geq 2^{j-k}}\|_{L^2}^2=\frac{1}{2\pi}\int_{\bR\setminus[-2^{j-k},2^{j-k}]}|\widehat{f}(\omega)|^2\, d\omega\lesssim \frac{1}{N^\beta}
\end{equation}
for every $\beta>0$, because $2^{j-k}\geq 2^{-2k-2}N$ and $\widehat{f}$ has rapid decay. 

Note that $f=\fm+\fii+\fp$ by construction, and a similar decomposition holds for $g$. Since $S_J$ is nonexpansive, the triangle inequality allows us to write
\begin{align}\label{eq cont 18}
\adm{S_J(L_\tau f)-S_J(f)}\geq & \adm{S_J(\fm+\fp+\gm+\gp)-S_J(\fm+\fp)}\\
&-2\| \fii\|_{L^2}-\|\gii\|_{L^2}. \nonumber
\end{align}
We stress that the last two terms are $O(N^{-\beta})$ for every $\beta>0$ by \eqref{eq cont 6} and \eqref{eq cont 8}. 

On the other hand, in view of \eqref{eq cont 1} we have that the functions $\fm+\gm$ and $\fp+\gp$ are separated in the wavelet domain, in the sense of Lemma \ref{lemma sep}. Therefore,
\[
S_J[\cP_J](\fm+\gm+\fp+\gp)= S_J[\cP_J](\fm+\gm)+S_J[\cP_J](\fp+\gp)
\]
and similarly, 
\[
S_J[\cP_J](\fm+\fp)= S_J[\cP_J](\fm)+S_J[\cP_J](\fp).
\]
To conclude, since $S_J[\cP_J]$ is norm preserving as clarified at the beginning of the proof, we have 
\begin{align*}
&\adm{S_J[\cP_J](\fm+\gm+\fp+\gp)-S_J[\cP_J](\fm+\fp)}\\
&\qquad\qquad\qquad\geq \|\gp\|_{L^2}-\|\gm\|_{L^2}-2\|\fp\|_{L^2}\\
&\qquad\qquad\qquad\gtrsim \frac{1}{N}
\end{align*}
for $N$ large enough, by \eqref{eq cont 6}, \eqref{eq cont 7} and \eqref{eq cont 8}.

Combining the last bound with \eqref{eq cont 18} finally gives \eqref{eq cont 14}, provided that $N$ is large enough. This concludes the proof of \eqref{eq cont 0}.  
The proof of \eqref{eq cons0} turns out to be an immediate consequence of the validity of \eqref{eq cont 0}, because  $\|f_n\|_{L^2}=\|f\|_{L^2}$ for every $n$ and, for $0\leq \alpha<1$, $\|\tau_n\|_{C^\alpha}\lesssim 2^{-n(1-\alpha)}$ as a result of \eqref{eq interp}.
\end{proof}

The following result provides lower bounds in the same spirit of \eqref{eq cons0} but for scattering norms instead of the $L^2$-norm. The factors $\max\{J'+n,1\}^{1/2}$ and $\max\{J'+n,1\}^{\beta}$
occurring in \eqref{eq cont 23} and \eqref{eq cont 23 bis} below essentially counteract the growth of the norms $\|U[\cP_{J'}]f_n\|$ and $\|U[\cP_{J'}]f_n\|_1$, respectively (cf.\ Proposition
\ref{pro embed}).

\begin{proposition}\label{pro pro0}
Under the same assumption (and notation) as in Theorem \ref{thm mainthm0} we have the following lower bounds. 

There exists $C>0$ such that for every $J,J'\in\bZ$, $n\in\bN$,
\begin{equation}\label{eq cont 23}
\|S_J[\cP_J](L_{\tau_n} f_n)-S_J[\cP_J](f_n)\|\geq  \frac{C 2^{n(1-\alpha)}}{\max\{J'+n,1\}^{1/2}}\|\tau_n\|_{C^\alpha}\|U[\cP_{J'}]f_n\|.
\end{equation}

Moreover, for every $\beta>1$ there exists $C>0$ such that, for every $J,J'\in\bZ$, $n\in\bN$,
\begin{equation}\label{eq cont 23 bis}
\|S_J[\cP_J](L_{\tau_n} f_n)-S_J[\cP_J](f_n)\|\geq  \frac{C 2^{n(1-\alpha)}}{\max\{J'+n,1\}^{\beta}}\|\tau_n\|_{C^\alpha}\|U[\cP_{J'}]f_n\|_1.
\end{equation}
\end{proposition}
\begin{proof} 
The proof of \eqref{eq cont 23} is carried out along the lines of that of \eqref{eq cons0} in Theorem \ref{thm mainthm0}, now using that \begin{equation}\label{eq cont 20}
\|U[\cP_{J'}]f_n\|\lesssim \max\{J'+n,1\}^{1/2}.
\end{equation}
Indeed, as already observed, under the assumptions of Theorem \ref{thm mainthm0} the hypotheses of Proposition \ref{pro embed} are satisfied as well, so that the latter bound can be inferred from \eqref{eq embell2} and a suitable change of variable: \[
\|U[\cP_{J'}]f_n\|\lesssim\Big(\int_\bR |\widehat{f}(\omega)|^2 \log(e+2^{J'+n}|\omega|)\, d\omega \Big)^{1/2}\lesssim \max\{J'+n,1\}^{1/2},
\]
where we also used that $e+2^{J'+n}|\omega|\leq (e+2^{J'+n})(e+|\omega|)$. 

As far as \eqref{eq cont 23 bis} is concerned, it is just enough to replace \eqref{eq embell2} with \eqref{eq embell1} in the aforementioned arguments.
\end{proof}

\begin{remark}
We emphasize that letting $J'\to-\infty$ in \eqref{eq cont 23}, for fixed $J$ and $n$, yields \eqref{eq cont 0}. Indeed, we have that $\|U[\cP_{J'}]f\|\to \|f\|_{L^2}$ if $f\in L^2(\rd)$ and  $\|U[\cP_{J_0}]f\|<\infty$ for some $J_0\in\bZ$. This follows by an application of the dominated convergence theorem to the map $p\mapsto\|U[p]f\|^2_{L^2}\mathbbm{1}_{\cP_{J'}}(p)$ on the set $\cP_{J_0}$ (the counting measure being understood), since $\mathbbm{1}_{\cP_{J'}}\to \mathbbm{1}_{\{\emptyset\}}$ pointwise.

In addition, note that the quantity $\|S_J[\cP_J](L_\tau f)-S_J[\cP_J](f)\|$, for fixed $f,\tau$, is nonincreasing when $J$ increases -- see Section \ref{sec a brief review} and \cite[Proposition 2.9]{mall cpam} in this connection. On the contrary, both $\|U[\cP_{J'}]f\|$ and $\|U[\cP_{J'}]f\|_1$ are increasing with $J'$.  
These observations show that the results in Theorem \ref{thm mainthm0} and Proposition \ref{pro pro0} are particularly interesting in the regime $J\gg1$.
\end{remark}

We conclude this section by providing a lower bound for the modulus of continuity of the map $\tau\mapsto S_J[\cP_J]L_\tau f$, for signals $f$ with finite scattering norms. 
The proof is omitted, as it ultimately relies on the same arguments already used for proving Proposition \ref{pro pro0}. 

\begin{proposition}
Consider $0\leq\alpha<1$. Under the same assumptions (and notation) of Theorem \ref{thm mainthm0}, there exists a constant $C>0$ such that, for every $J,J'\in\bZ$ and for every $n$ large enough,
\[
\|S_J[\cP_J](L_{\tau_n}f_n)-S_J[\cP_J](f_n)\|\geq C \log^{-1/2}\big(\|\tau_n\|^{-1}_{C^\alpha}\big)\|U[\cP_{J'}]f_n\|.\]
Similarly, for every $\beta>1$ we have 
\[
\|S_J[\cP_J](L_{\tau_n}f_n)-S_J[\cP_J](f_n)\|\geq C  \log^{-\beta}\big(\|\tau_n\|^{-1}_{C^\alpha}\big)\|U[\cP_{J'}]f_n\|_1.
\]
\end{proposition}

\section{Stability results for $C^{1+\alpha}$ regularity, $ \alpha>0$ }\label{sec rp}
In this section we provide the proof of Theorem \ref{thm mainthm1}, where we consider deformation fields $\tau \in C^{1+\alpha}(\rd;\rd)$ for $0< \alpha < 1$. Recall that this implies $\tau \in L^\infty(\rd;\rd)$ and $D\tau \in C^\alpha(\rd;\rd)$. We assume that $\phi$ and $\psi$ are scattering filters in the sense of Section \ref{sec a brief review}.


The proof of Theorem \ref{thm mainthm1} mostly relies on the structure of \cite[Theorem 2.12]{mall cpam} and the ancillary results \cite[Lemmas 2.13 and 2.14]{mall cpam}. For our purposes we only need to isolate a limited number of key steps and elaborate on those as detailed below. Nevertheless, let us briefly discuss the complete roadmap for the sake of clarity. 

The goal is to bound the quantity $\| S_J[\cP_J](L_\tau f) - S_J[\cP_J](f) \|$ in terms of $\| \tau \|_{\Ca}$ and a scattering norm of $f$. It is clear that
\[ \| S_J[\cP_J](L_\tau f) - S_J[\cP_J](f) \| \le \| L_\tau(S_J[\cP_J] f) - S_J[\cP_J](f) \| + \| [S_J[\cP_J],L_\tau]\|,\] where the commutator of two operators $A,B$ is defined by $[A,B] = AB - BA$. 

Let us focus on the first term. The stabilizing properties of the average $A_J$ underlying $S_J[\cP_J]$ can be exploited to prove that
\[ \| L_\tau(S_J[\cP_J] f) - S_J[\cP_J](f) \| \lesssim 2^{-J} \|\tau \|_{L^\infty} \| U[\cP_J] f \|, \] see \cite[(2.42) and (2.51)]{mall cpam} -- which actually hold assuming only $\tau \in C^1(\rd;\rd)$. 

Controlling the commutator error $\|[S_J[\cP_J],L_\tau]\|$ is actually the main difficulty of this result. First, it is proved in \cite[Lemma 2.13]{mall cpam} that for any operator $L$ on $L^2(\rd)$ one has
\[ \| [S_J[\cP_J],L] f\| \le \| [U_J,L] \| \| U[\cP_J] f \|_1, \] where $U_J h = \{ A_J h, U[\Lambda_J]h\}$ for $h \in L^2(\rd)$. Note that $U_J = MW_J$, where $M$ is the nonexpansive operator on $\ell^2(\{J\} \cup \Lambda_J;L^2(\rd))$ given by $M\{ h_J, (h_\lambda)_{\lambda \in \Lambda_J}\} = \{ h_J, (|h_\lambda|)_{\lambda \in \Lambda_J}\}$, and since $[M,L_\tau]=0$ the problem ultimately reduces to bounding the commutator error $\| [W_J,L_\tau]\|$ between the Littlewood-Paley wavelet transform $W_J$ at scale $2^J$ and the deformation operator $L_\tau$. For $\tau \in C^2(\rd;\rd)$ it was proved in \cite[Lemma 2.14]{mall cpam} that
\[ \| [W_J,L_\tau]\| \lesssim 2^{-J}\|\tau\|_{L^\infty} + \badmax \|D\tau\|_{L^\infty} + \| D^2\tau \|_{L^\infty}. \]  
The proof of this result is a technical \textit{tour de force} among quite delicate estimates. Roughly speaking, the operator $[W_J,L_\tau]^*[W_J,L_\tau]$ has a singular kernel along the diagonal, and the standard method of harmonic analysis suggests considering a suitable frequency decomposition. The singular part of the operator is then carried by the high-frequency terms, and the latter are eventually bounded using the Cotlar lemma. 

Our contribution in this connection is an improvement of the estimate above.

\begin{proposition}\label{prop commutator}
Let $0<\alpha<1$. There exists a constant $C>0$ such that for all $J \in \bZ$ and $\tau \in C^{1+\alpha}(\rd;\rd)$, with $\| D\tau \|_{L^\infty} \le 1/2$,
\begin{equation}\label{eq main comm bound}
    \| [W_J,L_\tau]\| \le C \lc 2^{-J}\|\tau\|_{L^\infty} + \badmax \|D\tau\|_{L^\infty} + |D\tau|_{\Ca}\rc. 
\end{equation}
\end{proposition}

\begin{proof}

We follow the same pattern of the proof of \cite[Lemma 2.14]{mall cpam}, taking for granted all the estimates proved there under the assumption $\tau \in C^1(\rd;\rd)$. We thus confine ourselves to describe the main strategy, while focusing on the necessary modifications. We adhere to the notation used in Mallat's proof for the convenience of the reader. 

First, the problem is recast as follows:
\begin{align}\label{eq main comm est}
    \| [W_J,L_\tau] \| & = \| [W_J,L_\tau]^*[W_J,L_\tau] \|^{1/2} \nonumber \\
    & \le \underbrace{\sum_{r \in G^+} \Big\Vert \sum_{j=-J+1}^\infty [W[2^jr],L_\tau]^*[W[2^jr],L_\tau]   \Big\Vert^{1/2}}_{\eqqcolon I_1} + \underbrace{\| [A_J,L_\tau]^* [A_J,L_\tau] \|^{1/2}}_{\eqqcolon I_2}. 
\end{align}
Bounds for the latter quantities are derived in \cite[Lemma E.1]{mall cpam}, that will be improved as well in accordance with our weaker regularity assumptions. The main argument goes as follows. For $j \in \bZ$ consider 
\begin{equation} \label{eq def Zj}
Z_j f = f*h_j, \qquad h_j(x)=2^{dj}h(2^j x), \end{equation} for a twice differentiable function $h \colon \rd \to \bC$ that decays like $O((1+|x|)^{-d-3})$ along with all its first- and second-order partial derivatives. We introduce the companion operators $K_j \coloneqq Z_j - L_\tau Z_j L_\tau^{-1}$ and note that $K_j = K_{j,1} + K_{j,2}$, where the latter are integral operators with kernels respectively given by
\begin{equation} \label{eq ker kj1} k_{j,1}(x,u) = 2^{dj} g(u,2^j(x-u)), \end{equation}
\[ k_{j,2}(x,u) = \det(I-D\tau(u)) \lc h_j((I-D\tau(u))(x-u)) - h_j(x-\tau(x)-u+\tau(u))\rc, \] where we set
\begin{equation}\label{eq def g} g(u,v) \coloneqq h(v)-h\big((I-D\tau(u)) v \big) \det(I-D\tau(u)), \quad (u,v) \in \rdd. \end{equation}

\par\medskip{\noindent \bf Step 1. Bound for $I_2$.} \par
In light of the previous assumptions, we have
\[ \|[Z_j,L_\tau]\| = \|[Z_j,L_\tau]^*[Z_j,L_\tau]\|^{1/2} \le \|L_\tau \| \| K_j^* K_j \|^{1/2} = \|L_\tau\| \|K_j\|\leq 2^{d/2}\|K_j\|. \] 
A bound for $I_2= \| [A_J,L_\tau]\|$ can thus be obtained by bounding $\|K_j\|$ in the case where $h=\phi$ and $j=-J$. In particular, since
\begin{equation}\label{eq Kj sum} \| K_j \| \le \|K_{j,1} \| + \|K_{j,2}\|,\end{equation} it is enough to separately control the norms of the latter integral operators. 

\par\medskip{\noindent \bf 1.1. Bound for $\|K_{j,1}\|$.} \par This was already proved in \cite[Eq. (E.18)]{mall cpam}: 
\begin{equation}\label{eq bound Kj1}
    \| K_{j,1} \| \lesssim \| D\tau \|_{L^\infty}. 
\end{equation}

\par\medskip{\noindent \bf 1.2. Bound for $\| K_{j,2}\|$.} \par

Consider the kernel $k_{j,2}$. A Taylor expansion of $\tau(x)$ centered at $u$ gives
\begin{align*}
    \tau(x)- \tau(u) & = \int_0^1 D\tau(u+t(x-u))(x-u) dt \\ 
    & = D\tau(u) (x-u) + \int_0^1 \big( D\tau(u+t(x-u)) - D\tau(u)\big) (x-u) dt \\ & = D\tau(u)(x-u) + \alpha(u,x-u),
\end{align*} where in view of the assumption on $\tau$ the remainder $\alpha(u,x-u)$ defined by the above formula satisfies 
\begin{equation} \label{eq rem bound}
|\alpha(u,x-u)| \le |D\tau|_{\Ca} |x-u|^{1+\alpha}.    
\end{equation}
Combining this result with a Taylor expansion of $h_j$ finally gives
\begin{multline*}
 k_{j,2}(x,u) = -\det(I-D\tau(u)) \int_0^1 Dh_j \big((I-tD\tau(u))(x-u) \\ + (1-t)(\tau(u)-\tau(x))\big) \alpha(u,x-u) dt. 
\end{multline*} 

We now infer a bound for $\|K_{j,2}\|$ by controlling the norm of the kernel in view of Schur's lemma. First, note that the assumption $\|D\tau\|_{L^\infty} \le 1/2$ implies $|\det(I-D\tau(y))| \le 2^d$. Moreover, using $Dh_j(y) = 2^{j(d+1)}Dh(2^jy)$ and the substitution $x'=2^j(x-u)$ we obtain
\begin{multline*}
 \ird |k_{j,2}(x,u)|dx \le 2^d \ird \bigg| \int_0^1 Dh \big((I-tD\tau(u))x' \\ + (1-t)2^j(\tau(u)-\tau(2^{-j}x'+u))\big) 2^j\alpha(u,2^{-j}x') dt \bigg| dx'. 
\end{multline*} Recall that $|Dh(y)| \lesssim (1+|y|)^{-d-3}$ by assumption, and it is easy to see that for $0\le t \le 1$ we have
\begin{align*}
 \big| (I-tD\tau(u))x' + (1-t)2^j(\tau(u)-\tau(2^{-j}x'+u))\big|  & \ge |x'|(1-\|D\tau\|_{L^\infty}) \\ & \ge |x'|/2.
\end{align*}
Concerning the term $|2^j \alpha(u,2^{-j}x')|$, on the one hand, by virtue of \eqref{eq rem bound} we have  
    \[ |2^j \alpha(u,2^{-j}x')| \lesssim 2^{-j\alpha} |D\tau|_{\Ca} |x'|^{1+\alpha}.  \] 
    Then
    \[ \ird |k_{j,2}(x,u)|dx  \lesssim  2^{-j\alpha} |D\tau|_{\Ca} \ird (1+|x'|/2)^{-d-3} |x'|^{1+\alpha} dx'  \lesssim 2^{-j\alpha}|D\tau|_{\Ca}. \]
    
    On the other hand, we observe that
    \[ |2^j \alpha(u,2^{-j}x')| =2^j |\tau(2^{-j}x' + u) - \tau(u)-D\tau(u) (2^{-j}x')| \le  2 \|D\tau\|_{L^\infty} |x'|,\] 
    resulting in
    \[ \ird|k_{j,2}(x,u)|dx  \lesssim \|D\tau\|_{L^\infty}  \ird (1+|x'|/2)^{-d-3} |x'| dx' \lesssim \|D\tau\|_{L^\infty}. \]
Combining the previous estimates gives
\[ \ird |k_{j,2}(x,u)|dx \lesssim \min\{ 2^{-j\alpha}|D\tau|_{\Ca}, \|D \tau\|_{L^\infty}\}, \] and the same bound holds for $\ird |k_{j,2}(x,u)|du$ since the previous arguments apply in the same form after the substitution $u'=2^j(x-u)$. As a consequence of Schur's lemma we thus obtain
\begin{equation} \label{eq bound Kj2} \|K_{j,2} \| \lesssim \min\{ 2^{-j\alpha}|D\tau|_{\Ca}, \|D \tau\|_{L^\infty}\}. \end{equation}
The combination of the estimates above shows that the term $I_2$ in \eqref{eq main comm est} is $\lesssim \|D\tau\|_{L^\infty}$.

\par\medskip{\noindent \bf Step 2. Bound for $I_1$.} \par
Consider again the convolution operator $Z_j$ introduced in \eqref{eq def Zj}. By mimicking \cite[Lemma E.1]{mall cpam}, we will prove that if $\ird h(x)dx=0$ then
\begin{equation}\label{eq sum diagonal}
    \lV \sum_{j =-\infty}^{+\infty} [Z_j,L_\tau]^* [Z_j,L_\tau] \rV^{1/2} \lesssim \badmax \| D\tau\|_{L^\infty} + |D\tau|_{\Ca}. 
\end{equation} Consider in particular the case where $h(x)=\psi(r^{-1}x)$ for each $r \in G^+$ and replace $-\infty$ with $-J+1$ in the summation ($[Z_j,L_\tau]^* [Z_j,L_\tau]$ is a positive operator). The resulting bound, combined with the one for $I_2$ proved above, will conclude the proof of \eqref{eq main comm bound} and thus of Proposition \ref{prop commutator}.

To this aim, let us first remark that the nonsingular part of the commutator has been isolated and bounded in \cite[Pages 1386-1387]{mall cpam}, yielding
\begin{equation}\label{eq sum diagonal 2}
\lV \sum_{j =-\infty}^{+\infty} [Z_j,L_\tau]^* [Z_j,L_\tau] \rV^{1/2} \lesssim \badmax \| D\tau\|_{L^\infty} + \lV \sum_{j = 0}^{+\infty} K_j^* K_j \rV^{1/2}. 
\end{equation}


\par\medskip{\noindent \bf 2.1. Bound for $\big\| \sum_{j=0}^{+\infty}
K_j^* K_j \big\|^{1/2}$.}\par
Recall that $K_j=K_{j,1}+K_{j,2}$, hence
\begin{align} \label{eq sum diagonal 3}\lV \sum_{j=0}^{+\infty} K_j^* K_j \rV^{1/2} & \le \lV \sum_{j=0}^{+\infty} K_{j,1}^* K_{j,1} \rV^{1/2} \\
& + \sum_{j=0}^{+\infty} \lc \|K_{j,2}\| + 2^{1/2}\|K_{j,1}\|^{1/2}\|K_{j,2}\|^{1/2}\rc.\nonumber\end{align}
Using \eqref{eq bound Kj1} and \eqref{eq bound Kj2} above, we have that
\begin{equation}\label{eq 4.12bis}
 \sum_{j=0}^{+\infty} \lc \|K_{j,2}\| + 2^{1/2}\|K_{j,1}\|^{1/2}\|K_{j,2}\|^{1/2}\rc \lesssim \|D\tau\|_{L^\infty} + |D\tau|_{\Ca}. 
 \end{equation}

\par\medskip{\noindent \bf  2.2. Bound for $\big\| \sum_{j=0}^{+\infty} K_{j,1}^* K_{j,1} \big\|^{1/2}$.} \par 
Recall that here we are assuming that $\ird h(x)dx=0$. 

The goal of this section is to provide a bound for $\big\| \sum_{j=-\infty}^{+\infty} Q_j \big\|$, where we set $Q_j=K_{j,1}^*K_{j,1}$ for $j\geq0$ and $Q_j =0$ for $j<0$. We will apply Cotlar's lemma \cite[Chapter VII]{stein}: if there is a sequence of positive real numbers $\{\beta(j)\}_{j\in \bZ}$ such that $\sum_{j\in \bZ} \beta(j) < \infty$ and
\[ \| Q_j^* Q_l \| \le \beta(j-l)^2, \quad \| Q_j Q_l^* \| \le \beta(j-l)^2, \] then
\[ \lV \sum_{j = -\infty}^{+\infty} Q_j \rV \le \sum_{j = -\infty}^{+\infty} \beta(j).\] 
As a consequence of the self-adjointness of $Q_j$ it is enough to provide a bound for $\|Q_l Q_j\|$ only, hence we resort again to Schur's lemma.  

Let  $\bar{k}_{l,j}$ be the integral kernel of $Q_l Q_j$. Using \eqref{eq ker kj1} we obtain
\begin{multline}
    \ird |\bar{k}_{l,j}(y,z)|dy = \ird \Big| \ird g(u,x)g(u,x')2^{dl}\overline{g(y,x+2^l(u-y))} \\ \times 2^{dj}\overline{g(z,x'+2^j(u-z))}dxdx'du \Big| dy.
\end{multline} 
Therefore, we need to prove a suitable bound for the functional
\begin{multline} \varphi(G) = \varphi_{x,x',z,j,l}(G) \coloneqq \ird \Big| \ird G(u) 2^{dl}\overline{g(y,x+2^l(u-y))} \\ \times 2^{dj}\overline{g(z,x'+2^j(u-z))}du \Big| dy, 
\end{multline} for $G \in \Ca(\rd)$.
We restrict to the case $j\ge l\ge 0$ in view of the symmetric role of these parameters.
We will prove below that
\begin{equation} \label{eq bound varphiG} \varphi(G) \lesssim 2^{(l-j)\alpha} \| D\tau\|_{L^\infty}^2 \|G\|_{\Ca}, \end{equation} where the hidden constant does not depend on $x,x',z$. We are going to apply this estimate to $G(u) \coloneqq g(u,x)g(u,x')$, where here $x,x'\in\rd$ play the role of parameters. We then will  prove that
\begin{equation}\label{eq 30}
\|G\|_{\Ca}\lesssim (1+|x|)^{-d-1}(1+|x'|)^{-d-1} \| D\tau \|_{L^\infty}\big(\| D\tau \|_{L^\infty} + |D\tau|_{\Ca} \big).
\end{equation}

The latter bounds allow us to conclude that 
\begin{align*} \ird |\bar{k}_{l,j}(y,z)|dy  & \leq \ird \varphi(g(\cdot,x)g(\cdot,x'))\,dx\, dx' \\
& \lesssim 2^{(l-j)\alpha}\big( \|D\tau\|_{L^\infty}^4 + \|D\tau\|_{L^\infty}^3 |D\tau|_{\Ca} \big) \\
& \lesssim 2^{(l-j)\alpha}\big( \|D\tau\|_{L^\infty}+  |D\tau|_{\Ca} \big)^4. 
\end{align*}
The same arguments lead one to the same bound for $\ird |\bar{k}_{l,j}(y,z)|dz$. Therefore, by Schur's lemma we have
\[ \| Q_l Q_j \| \lesssim  2^{(l-j)\alpha}\big( \|D\tau\|_{L^\infty}+  |D\tau|_{\Ca} \big)^4, \]
and Cotlar's lemma with $\beta(j) = C 2^{-|j|\alpha/2}  \big( \|D\tau\|_{L^\infty}+  |D\tau|_{\Ca} \big)^2$ for a suitable constant $C>0$ finally implies that \[ \lV \sum_{j = 0}^{+\infty} K_{j,1}^*K_{j,1} \rV \lesssim \big( \|D\tau\|_{L^\infty}+  |D\tau|_{\Ca} \big)^2,
\]
which combined with \eqref{eq sum diagonal 2}, \eqref{eq sum diagonal 3} and \eqref{eq 4.12bis} 
gives \eqref{eq sum diagonal} and concludes the proof.

The proofs of \eqref{eq bound varphiG} and \eqref{eq 30} are given below. 
\end{proof}

\begin{proof}[Proof of \eqref{eq bound varphiG}]

A Taylor expansion of $h$ in \eqref{eq def g} (cf.\ \cite[(E.30)]{mall cpam}) implies that 
\begin{multline} \label{eq tildeg}
    g(u,v) =  \big( 1-\det (I-D\tau(u))\big) h\big( (I-D\tau(u))v \big) \\
    + \int_0^1 Dh\big( (1-t)v+t(I-D\tau(u))v \big) \cdot D\tau(u) v dt.
\end{multline}
Using that $\det (I-D\tau(u)) \ge (1- \|D\tau\|_{L^\infty})^d$,  $\|D\tau\|_\infty\le 1/2$ and the fact that $h(x)$ and $Dh(x)$ by assumption decay like $(1+|x|)^{-d-3}$, we obtain (cf.\ \cite[(E.31)]{mall cpam}) that
\begin{equation}\label{eq mall E31} |g(u,v)| \lesssim \|D\tau\|_{L^\infty} (1+|v|)^{-d-2}.
\end{equation} 
Using this estimate it is easy to infer that for any $G \in C^0(\rd)$ we have
\[ \varphi(G) \lesssim \|D\tau\|_{L^\infty}^2 \|G\|_{C^0}.\] Moreover, the assumption $\int_{\rd} h(x)\, dx=0$ (implying $\int_{\rd} g(u,v)\, dv=0$ for all $u$) and \eqref{eq mall E31} imply that we can write $g(u,v)=\partial \overline{g}(u,v)/\partial v_1$, with $\overline{g}(u,v)$ satisfying (cf. \cite[(E.37)]{mall cpam}) 
\[
|\overline{g}(u,v)|\leq C\|D\tau\|_{\infty}(1+|v|)^{-d-1},
\]
and from \eqref{eq tildeg} we have (cf.\ \cite[(E.39)]{mall cpam}) 
\[
\Big|\frac{\partial g(u,v)}{\partial v_1}\Big|\lesssim \|D\tau\|_{\infty}(1+|v|)^{-d-1}.
\]
Using these estimates, an integration by parts with respect to $u_1$ (cf. \cite[Page 1391]{mall cpam}) therefore yields
\begin{align*}
    \varphi(G) & \lesssim 2^{-j} \|D\tau\|_{L^\infty}^2 \| G \|_{C^1}+ 2^{l-j}\|D\tau\|_{L^\infty}^2 \| G \|_{C^0} \\ & \lesssim 2^{l-j}\|D\tau\|_{L^\infty}^2 \| G \|_{C^1}, 
\end{align*} for every $G\in C^1(\rd)$. Note that the functional $\varphi$ is subadditive, namely $|\varphi(f+g)| \le |\varphi(f)| + |\varphi(g)|$. By real interpolation (see e.g., \cite[Theorem 6]{maligranda}), since $[C^0,C^1]_{\alpha,\infty} = C^\alpha$ (see e.g., \cite[(1.16)]{lunardi}), we obtain \eqref{eq bound varphiG} for every $G \in C^\alpha(\rd)$. 
\end{proof}

\begin{proof}[Proof of \eqref{eq 30}]
Recall that $G(u)=g(u,x)g(u,x')$, where $g$ is given in \eqref{eq tildeg}.  
From the very definition of the $\Ca$ norm, using the fractional Leibniz rule \eqref{eq frac leib} and \eqref{eq mall E31} to bound the norm in $L^\infty(\rd)$ of $G$, it suffices eventually to show that  
\begin{equation}\label{eq 4.19}
    |g(\cdot,v)|_{\Ca} \lesssim (1+|v|)^{-d-1} \big(\| D\tau \|_{L^\infty} + |D\tau|_{\Ca} \big).
\end{equation}
We use the expression in \eqref{eq tildeg}, and again the fractional Leibniz rule \eqref{eq frac leib}. Precisely, the desired estimate  \eqref{eq 4.19} will follow from the bounds given below on the $L^\infty$ norm and $C^\alpha$ seminorm in $\rd$ (with respect to $u$) of the factors appearing in $\eqref{eq tildeg}$.

\begin{itemize} \setlength\itemsep{2ex}
    
    \item 
    Let $b_1(u)=1-\det (I-D\tau(u))$. Since $\det (I-D\tau(u)) \ge (1- \|D\tau\|_{L^\infty})^d$ we have
    \[ \|b_1\|_{L^\infty} \lesssim \|D\tau\|_{L^\infty}.\] Moreover, since $b_1$ is a polynomial in the entries of the matrix $D\tau(u)$ and $\|D\tau(u)\|_{L^\infty} \le 1/2$, we obtain
    \[ |b_1|_{\Ca} \lesssim |D\tau|_{\Ca} \] by a straightforward application of Schauder's estimate \eqref{eq schaud}.
    
    \item Define 
    $b_2(u,v)=h\big( (I-D\tau(u)) v\big)$. Clearly
    \[ \|b_2(\cdot,v) \|_{L^\infty} \lesssim (1+|v|)^{-d-3}. \] A bound for $|b_2(\cdot,v)|_{\Ca}$ can be obtained using Schauder's estimate \eqref{eq schaud}, in particular 
    \[ |b_2(\cdot,v)|_{\Ca} \lesssim |h|_{\Lip(B_v)} |\wt{b_2}(\cdot,v)|_{\Ca},\]
    where we introduced the companion map $\wt{b_2}(u,v)=(I-D\tau(u)) v$ and $B_v$ stands for the range of the map $\tilde{b}_2(\cdot,v)$, for fixed $v$. 
    
    First, we remark that
     \[ |\wt{b_2}(\cdot,v)|_{\Ca} \lesssim |D\tau|_{\Ca}|v|. 
     \] Moreover, since $Dh(u) = O((1+|u|)^{-d-3})$ by assumption, and since $B_v$ is contained in the ball $B(v,|v|/2)$, we have
     \[
     |h|_{\Lip(B_v)} \lesssim (1+|v|)^{-d-3}. \] 
   We then conclude that
    \[ |b_2(\cdot,v)|_{\Ca} \lesssim |D\tau|_{\Ca} (1+|v|)^{-d-2} \] as a combination of the previous results. 
    \item Lastly, consider the map $b_3$
    defined for $0 \le t \le 1$ by \[ b_3(u,v)= Dh \big( I-t D\tau(u))v \big) \cdot D\tau(u)v. \]
    
    An application of the fractional Leibniz formula \eqref{eq frac leib} combined with the same estimates for $b_2$ above with $Dh$ in place of $h$ finally gives
    \begin{align*}
        |b_3(\cdot,v)|_{\Ca} &\lesssim (1+|v|)^{-d-2} |D\tau|_{\Ca} + (1+|v|)^{-d-1}|D\tau|_{\Ca} \|D\tau\|_{L^\infty}\\
        &\lesssim (1+|v|)^{-d-1}|D\tau|_{\Ca}.
        \end{align*}
    \end{itemize}
This concludes the proof of  \eqref{eq 4.19}. 
\end{proof}
\begin{remark}\label{rem J1} It was already pointed out in \cite{mall cpam} that the term $\badmax$ in \eqref{eq sum diagonal 2} can be replaced by $\max\{J,1\}$ -- to be precise, it is enough to choose $\gamma=\max\{J,1\}$ in \cite[(E.7)--(E.10)]{mall cpam}. Therefore, we have also the estimate 
\begin{multline}\label{eq main fm1 bis}
        \| S_J[\cP_J](L_\tau f) - S_J[\cP_J](f) \| \\ \le  C \lc 2^{-J}\|\tau\|_{L^\infty} + \max\{J,1\} \|D\tau\|_{L^\infty} + |D\tau|_{\Ca}\rc \|U[\cP_J]f\|_1. 
\end{multline}
\end{remark}

\section{Stability up to $\varepsilon$ losses for Lipschitz deformations}
In this section we focus on Lipschitz deformations; in particular, we prove Theorem \ref{thm mainthm 3}. We continue to assume that $\phi$ and $\psi$ are scattering filters in the sense of Section \ref{sec a brief review}. 

We need some preliminary results from the theory of real interpolation of Besov spaces (cf.\ \cite{triebel}). 

Let $\phi_0\in C^\infty(\rd)$ be supported in the ball $|\omega|\leq 2$, with $\phi_0(\omega)=1$ for $|\omega|\leq 1$. Set  $\phi_j(\omega)=\phi_0(2^{-j} \omega)$, $j\in \bZ$. The functions $\tilde{\phi}_j\coloneqq\phi_j-\phi_{j-1}$, $j\in\bZ$, are supported in the annuli $2^{j-1}\leq|\omega|\leq 2^{j+1}$ and induce a Littlewood-Paley partition of unity of $\rd\setminus\{0\}$. Recall that the Besov norms with $s\in\bR$, $1\leq p,q\leq\infty$ are accordingly defined, for a temperate distribution $f$, by
\[
\|f\|_{B^s_{p,q}}\coloneqq\Big(\|\phi_0(D) f\|^q_{L^p}+\sum_{j\geq 1}2^{jsq}\|\tilde{\phi}_j(D)f\|_{L^p}^q\Big)^{1/q},
\] where $\tilde{\phi}_j(D)=\cF^{-1}\tilde{\phi}_j\cF$ stands for the Fourier multiplier with symbol $\tilde{\phi}_j$, $j \in \bZ$, and similarly $\phi_0(D)=\cF^{-1}\phi_0\cF$. Obvious changes are needed if $q=\infty$. 

Recall that for $s>0$ \textit{not} integer, the space $B^s_{\infty,\infty}(\rd)$ coincides with the H\"older class $C^s(\rd)$ considered above. If $s=1$ then $B^1_{\infty,\infty}(\rd)$ \textit{contains} the space ${\rm Lip}(\rd)$ of bounded Lipschitz function $f\colon \rd\to\bC$, endowed with the norm  
\[
\|f\|_{{\rm Lip}}\coloneqq\|f\|_{L^\infty}+|f|_{{\rm Lip}}= \|f\|_{L^\infty}+\|\nabla f \|_{L^\infty},
\]
where $\nabla f$ is understood in the sense of distributions or even almost everywhere.  


With a temperate distribution $f$ we also associate the set 
\begin{equation}\label{eq def A}
\cA_f\coloneqq \{\phi_j(D) f:\ j\geq 0\}.
\end{equation}
Observe that the operators $\phi_j(D)$, $j\in \bZ$, are uniformly bounded on $L^\infty(\rd)$, since $\cF^{-1}\phi_0\in L^1(\rd)$. 

\begin{lemma}\label{lem inter}

Let $0<\theta<1$ and $s>1$ be such that $1=(1-\theta)s$. For $f\in B^0_{\infty,1}(\rd)$ and $t>0$, consider 
\[
\tilde{K}(t,f,B^s_{\infty,\infty},B^0_{\infty,1})\coloneqq\inf_{f_0\in \cA_f} \{\|f_0\|_{B^s_{\infty,\infty}}+t\|f-f_0\|_{B^0_{\infty,1}}\}.
\]
There exists a constant $C>0$ such that, for every $f\in B^1_{\infty,\infty}(\rd)$, 
\begin{equation}\label{eq inter k}
\sup_{t>1} t^{-\theta} \tilde{K}(t,f,B^s_{\infty,\infty},B^0_{\infty,1})\leq C \|f\|_{B^1_{\infty,\infty}}.
\end{equation}
\end{lemma}
\begin{proof}
The functional $\tilde{K}(t,f,B^s_{\infty,\infty},B^0_{\infty,1})$ is just a variant of the $K$-functional in real interpolation theory (cf.\ \cite[Section 2.4.1]{triebel}), defined for $t>0$ and $f\in B^0_{\infty,1}(\rd)$ by 
\[
K(t,f,B^s_{\infty,\infty},B^0_{\infty,1})\coloneqq\inf \{\|f_0\|_{B^s_{\infty,\infty}}+t\|f-f_0\|_{B^0_{\infty,1}}: \  f_0\in B^s_{\infty,\infty}\}.
\]
It is well known that 
\[
\sup_{t>0} t^{-\theta}K(t,f,B^s_{\infty,\infty},B^0_{\infty,1})\lesssim \|f\|_{B^1_{\infty,\infty}}
\]
for every $f\in B^1_{\infty,\infty}(\rd)$, which amounts to the embedding $B^1_{\infty,\infty}\hookrightarrow (B^s_{\infty,\infty},B^0_{\infty,1})_{\theta,\infty}$. A proof of this fact can be found in \cite[Section 2.4.2]{triebel}, and an accurate inspection of the latter (the part dealing with $t>1$, to be precise) allows one to realize that \eqref{eq inter k} holds indeed. 
\end{proof}

We are now ready to prove Theorem \ref{thm mainthm 3}. We will consider vector fields $\tau$ in the Besov space $B^s_{p,q}(\rd;\rd)$ (i.e., the components belong to $B^s_{p,q}(\rd)$), endowed with the norm
\[
\|\tau\|_{B^s_{p,q}}\coloneqq \sum_{k=1}^d \|\tau^{(k)}\|_{B^s_{p,q}}, \qquad \tau=(\tau^{(1)},\ldots, \tau^{(d)}).
\]

\begin{proof}[Proof of Theorem \ref{thm mainthm 3}]
We can suppose $J=0$ by virtue of a scaling argument (cf.\ \eqref{eq scal}). Indeed, the estimate \eqref{eq ideal esp} is invariant under the substitutions $J \to J-n$ ($n\in\bZ$), $f(x)\to 2^{n/2}f(2^n x)$, $\tau(x)\to 2^{-n}\tau(2^n x)$ and $R\to 2^n R$. We can also suppose that $\|D\tau\|_{L^\infty}\leq\varepsilon_0$, with $\varepsilon_0$ small enough (to be fixed later on), because for $\varepsilon_0<\|D\tau\|_{L^\infty}\leq 1/2$ the estimate \eqref{eq ideal esp} with $J=0$ holds due to the fact that $S_0[\cP_0]$ is nonexpansive on $L^2(\rd)$ and $\|L_\tau\|\leq 2^{d/2}$.

We already know from Remark \ref{rem J1} (with $J=0$) that, for every $s\in (1,2)$ and $\tau\in C^s(\rd;\rd)=B^s_{\infty,\infty}(\rd;\rd)$ with $\|D\tau\|_{L^\infty}\leq 1/2$, 
\[
\| S_0[\cP_0](L_\tau f) - S_0[\cP_0](f) \|\lesssim \|\tau\|_{B^s_{\infty,\infty}}\|U_0[\cP_0]f\|_1.
\]

The assumption \eqref{eq emb int} and the fact that $\widehat{f}$ is supported in the ball $|\omega|\leq R$ imply that
\begin{equation}\label{eq int 1}
   \| S_0[\cP_0](L_\tau f) - S_0[\cP_0](f) \|\lesssim \log^\beta(e+R)\|\tau\|_{B^s_{\infty,\infty}}\|f\|_{L^2}. 
\end{equation}
On the other hand, if $\tau_0,\tau_1\in {\rm Lip}(\rd;\rd)$ satisfy $\|D\tau_0\|_{L^\infty}\leq 1/2$ and $\|D\tau_1\|_{L^\infty}\leq 1/2$, a Taylor expansion yields
\begin{align*}
\|L_{\tau_0}f-L_{\tau_1}f\|_{L^2}&\leq \|\tau_0-\tau_1\|_{L^\infty} \int_0^1\|L_{(1-t)\tau_1+t\tau_0}(\nabla f)\|_{L^2}\, dt\\
&\lesssim\|\tau_0-\tau_1\|_{L^\infty}\|\nabla f\|_{L^2}
\end{align*}
where we used that $\|D((1-t)\tau_1+t\tau_0)\|_{L^\infty}\leq 1/2$. Since $\|\nabla f\|_{L^2}\lesssim (1+R)\|f\|_{L^2}$ and $S_0[\cP_0]$ is nonexpansive, we conclude that 
\begin{align}\label{eq int 2}
\| S_0[\cP_0](L_{\tau_0} f) - S_0[\cP_0](L_{\tau_1}f) \|
&\lesssim (1+R)\|\tau_0-\tau_1\|_{L^\infty}\|f\|_{L^2}\\
&\lesssim (1+R)\|\tau_0-\tau_1\|_{B^0_{\infty,1}}\|f\|_{L^2},
     \nonumber
\end{align}
where we used the embedding $B^0_{\infty,1}(\rd;\rd)\hookrightarrow L^\infty(\rd;\rd)$ -- see for instance \cite[Proposition 2.1]{sawano}.

We now resort to a nonlinear interpolation argument between \eqref{eq int 1} and \eqref{eq int 2}. Set $\cH=\ell^2(\cP_0,L^2(\rd))$ and, for fixed $f$ as above and $\tau\in {\rm Lip}(\rd;\rd)$ with $\|D\tau\|_{L^\infty}\leq 1/2$, consider 
\[
T_f(\tau)\coloneqq  S_0[\cP_0](L_\tau f) - S_0[\cP_0](f).
\]
Let $s\in(1,2)$ and $\theta\in (0,1)$ be such that $1=(1-\theta)s$ and consider $\tau\in {\rm Lip}(\rd;\rd)$ with $\|D\tau\|_{L^\infty}\leq \varepsilon_0$. For any  $\tau_0\in B^s_{\infty,
\infty}(\rd;\rd)$ with $\|D\tau_0\|_{L^\infty}\leq 1/2$ we have, by the triangle inequality,
 \eqref{eq int 1} and \eqref{eq int 2}, for $t\geq 1$,
\begin{align*}
\|T_f(\tau)\|_\cH&\leq
 \|T_f(\tau_0)\|_{\cH}+t\|T_f(\tau)-T_f(\tau_0)\|_{\cH}\\
&\lesssim \log^\beta(e+R)\|\tau_0\|_{B^s_{\infty,\infty}}\|f\|_{L^2}+t(1+R)\| \tau-\tau_0\|_{B^0_{\infty,1}}\|f\|_{L^2}\\
&= \sum_{k=1}^d\big( \log^\beta(e+R)\|\tau_0^{(k)}\|_{B^s_{\infty,\infty}}+t(1+R)\| \tau^{(k)}-\tau_0^{(k)}\|_{B^0_{\infty,1}}\big)\|f\|_{L^2}
\end{align*}
 where we expanded the Besov norms in terms of the components, that is $\tau=(\tau^{(1)}, \ldots, \tau^{(d)})$ and $\tau_0=(\tau^{(1)}_0, \ldots, \tau^{(d)}_0)$.
 
Consider now $\chi\in C^\infty(\rd)$, supported where $|\omega|\leq 1$, with $\chi(\omega)=1$ for $|\omega|\leq 1/2$, along with the corresponding Fourier multiplier $\chi(D)$. We write $\tau_0^{(k)}=\chi(D)\tau_0^{(k)}+(1-\chi(D))\tau_0^{(k)}$. By the triangle inequality we obtain
 \begin{align}\label{eq inter20}
 \|T_f(\tau)\|_\cH
&\lesssim \log^\beta(e+R) \sum_{k=1}^d \|\chi(D)\tau_0^{(k)}\|_{B^s_{\infty,\infty}}\|f\|_{L^2}\\
&+ \sum_{k=1}^d\big( \log^\beta(e+R)\|(1-\chi(D))\tau_0^{(k)}\|_{B^s_{\infty,\infty}}+t(1+R)\| \tau^{(k)}-\tau_0^{(k)}\|_{B^0_{\infty,1}}\big)\|f\|_{L^2}.\nonumber
\end{align}
Since $\|D\tau\|_{L^\infty}\leq \varepsilon_0$, if $\varepsilon_0$ is small enough and $\tau_0^{(k)}\in\cA_{\tau^{(k)}}$ (cf.\ \eqref{eq def A}) we have $\|D\tau_0^{(k)}\|_{L^\infty}\leq C\|D\tau^{(k)}\|_{L^\infty} \leq  1/(2\sqrt{d})$, implying in particular that $\|D\tau_0\|_{L^\infty}\leq 1/2$. Moreover 
\[
\|\chi(D)\tau_0^{(k)}\|_{B^s_{\infty,\infty}}\lesssim \|\tau_0^{(k)}\|_{L^\infty}\lesssim \|\tau^{(k)}\|_{L^\infty}.
\]
We also remark that, since the Fourier transform of $\tau^{(k)}-\tau_0^{(k)}$  is supported in the region where $|\omega|\geq 1$ (with reference to \eqref{eq def A}, we have indeed that $\phi_j(\omega)=1$ in the ball $|\omega|\leq 1$), 
\[
\tau^{(k)}-\tau_0^{(k)}=(1-\chi(D))\tau^{(k)}- (1-\chi(D))\tau_0^{(k)}.
\]

In light of the facts highlighted so far, now we take the infimum of \eqref{eq inter20} over $\tau_0^{(k)}\in\cA_{\tau^{(k)}}$. For $t\geq 1$ we obtain
\begin{align*}
\|T_f(\tau)\|_\cH&\lesssim
\log^\beta(e+R)  \|\tau\|_{L^\infty}\|f\|_{L^2}\\
&+\log^\beta(e+R) \sum_{k=1}^d \tilde{K}(t(1+R)/\log^\beta(e+R),(1-\chi(D))\tau^{(k)},B^{s}_{\infty,\infty},B^{0}_{\infty,1})\|f\|_{L^2}
\end{align*}
where the functional $\tilde{K}$ is defined in Lemma \ref{lem inter} and we used the fact that
\[
\cA_{(1-\chi(D))\tau^{(k)}}=\{(1-\chi(D))f_0:\ f_0\in\cA_{\tau^{(k)}}\}.
\]

Finally, by multiplying the latter estimate by $t^{-\theta}$ and then by taking the supremum for $t>1$, we obtain 
\begin{align*}
\|T_f(\tau)\|_{\cH}&\lesssim
\big(\log^\beta(e+R)  \|\tau\|_{L^\infty}+
\log^{\beta(1-\theta)}(e+R)(1+R)^\theta \sum_{k=1}^d \|(1-\chi(D))\tau^{(k)}\|_{B^{1}_{\infty,\infty}}\big)\|f\|_{L^2}\\
&= \big(\log^\beta(e+R)  \|\tau\|_{L^\infty}+\log^{\beta(1-\theta)}(e+R)(1+R)^\theta \|(1-\chi(D))\tau\|_{B^1_{\infty,\infty}}\big)\|f\|_{L^2},
\end{align*}
where in the first inequality we applied Lemma \ref{lem inter}\footnote{Precisely, since $(1+R)/\log^\beta(e+R)\geq c_0$ for some $c_0>0$ and \[
\tilde{K}(t(1+R)/\log^\beta(e+R),f,B^s_{\infty,\infty},B^0_{\infty,1})\leq \tilde{K}(t\max\{c_0^{-1},1\}(1+R)/\log^\beta(e+R),f,B^s_{\infty,\infty},B^0_{\infty,1}),
\]
we can resort to \eqref{eq inter k} with $t$ replaced by $t\max\{c_0^{-1},1\}(1+R)/\log^\beta(e+R)>1$.}.
 On the other hand, we have
\[
\|(1-\chi(D))\tau\|_{B^1_{\infty,\infty}}\lesssim \|D\tau\|_{B^0_{\infty,\infty}}\lesssim \|D\tau\|_{L^\infty}.\]
The second inequality is clear from the definition of the $B^0_{\infty,\infty}$ norm. 
The first inequality follows by a standard argument that we sketch here for the benefit of the reader. Let $\tilde{\phi}'$ be a smooth function in $\rd$, supported in the annulus $2^{-2}\leq|\omega|\leq 2^4$, with $\tilde{\phi}'(\omega)=1$ for $2^{-1}\leq|\omega|\leq 2$. Then we write $\tilde{\phi}_j(\omega)=\tilde{\phi}'(2^{-j}\omega)\tilde{\phi}_j(\omega)$ in the definition of the Besov norm, and we observe that the functions $2^j \tilde{\phi}'(2^{-j}\omega)\omega_k/|\omega|^2$, $k=1,\ldots,d$, can be written as $\phi''_k(2^{-j}\omega)$ for some $\phi''_k$ smooth with compact support. Hence their inverse Fourier transforms have $L^1$ norm uniformly bounded with respect to $j$, and the corresponding Fourier multipliers are thus uniformly bounded on $L^\infty(\rd)$. Similarly, if $\phi_0'$ is a smooth function supported in the ball $|\omega|\leq 4$, with $\phi_0'(\omega)=1$ for $|\omega|\leq 2$ we can write $\phi_0=\phi_0'\phi_0$ and observe that $(1-\chi(\omega))\phi'_0(\omega)\omega_k/|\omega|^2$ is a smooth function with compact support, and hence defines a Fourier multiplier bound on $L^\infty(\rd)$. 

The desired estimate \eqref{eq ideal esp} is then proved since $s$ can be chosen arbitrarily close to $1$, hence making in turn the exponent $\theta$ arbitrarily small. 
\end{proof}


\section*{Acknowledgements}

The authors wish to express their gratitude to Giovanni S.\ Alberti and Matteo Santacesaria for fruitful conversations on the topics of the manuscript. They also gratefully acknowledge insightful suggestions from the referee of a previously submitted version of this note.  


S. Ivan Trapasso is member of the Machine Learning Genoa (MaLGa) Center, Università di Genova. This material is based upon work supported by the Air Force Office of Scientific Research under award number FA8655-20-1-7027, as well as by Fondazione Compagnia di San Paolo.

The authors are members of the Gruppo Nazionale per l'Analisi Matematica, la Probabilità e le loro Applicazioni (GNAMPA) of the Istituto Nazionale di Alta Matematica (INdAM).

\end{document}